\documentclass[10pt,a4paper]{article}
\usepackage{amsmath,amssymb,amsfonts,amsthm}
\usepackage{graphicx}
\usepackage{float,graphicx,color}
\usepackage[all]{xy}

\newcommand{\rmd}{\mathrm{d}}

\newcommand{\rmH}{\mathrm{H}}
\newcommand{\rmL}{\mathrm{L}}

\newcommand{\bbC}{\mathbb{C}}

\newcommand{\bbH}{\mathbb{H}}

\newcommand{\bbL}{\mathbb{L}}

\newcommand{\bbN}{\mathbb{N}}

\newcommand{\bbR}{\mathbb{R}}

\newcommand{\bbZ}{\mathbb{Z}}

\newcommand{\frH}{\mathfrak{H}}

\newcommand{\calA}{\mathcal{A}}

\newcommand{\calD}{\mathcal{D}}

\newcommand{\calH}{\mathcal{H}}

\newcommand{\calJ}{\mathcal{J}}
\newcommand{\calK}{\mathcal{K}}

\newcommand{\calM}{\mathcal{M}}

\newcommand{\calT}{\mathcal{T}}

\newcommand{\calZ}{\mathcal{Z}}

\newcommand{\half}{\frac{1}{2}}

\newcommand{\herm}{\textsc{h}}

\DeclareMathOperator{\Div}{div}
\renewcommand{\div}{\Div}

\DeclareMathOperator{\grad}{grad}
\DeclareMathOperator{\curl}{curl}

\DeclareMathOperator{\re}{\mathfrak{Re}}

\DeclareMathOperator{\tr}{tr}

\renewcommand{\Re}{\re}

\newcommand{\hs}{{\scriptstyle \bigstar}}
\newcommand{\bs}{{\scriptscriptstyle \bullet}}

\newcommand{\infsup}[4]{\inf_{#1}\sup_{#2}\frac{#3}{#4}}
\newcommand{\cleq}{\preccurlyeq}
\newcommand{\cgeq}{\succcurlyeq}
\newcommand{\ceq}{\approx}

\newcommand{\lls}{[ \! [}
\newcommand{\rrs}{] \! ]}

\newcommand{\ljump}{\lls}
\newcommand{\rjump}{\rrs}

\newcommand{\alter}{\mathrm{Alt}}

\newcommand{\beq}{\begin{equation}}
\newcommand{\eeq}{\end{equation}}

\newcommand{\mapping}[4]{
\left\{
\begin{array}{rcl}
\displaystyle #1  &\to& #2 ,\\
\displaystyle #3  &\mapsto      & #4.
\end{array} \right.
}

\definecolor{mygray}{gray}{0.75}
\definecolor{myyellow}{rgb}{0.9, 0.8,0.1}
\definecolor{myred}{rgb}{0.7, 0.15, 0.15}
\definecolor{myblue}{rgb}{0.15, 0.15, 0.7}

\newcounter{point}[section]

\newtheorem{theorem}{Theorem}[section]
\newtheorem{lemma}[theorem]{Lemma}
\newtheorem{corollary}[theorem]{Corollary}
\newtheorem{proposition}[theorem]{Proposition}

\theoremstyle{definition}
\newtheorem{definition}{Definition}[section]

\theoremstyle{remark}
\newtheorem{remark}{Remark}[section]

\newcounter{tpoint}[theorem]

\theoremstyle{plain}

\theoremstyle{definition}

\theoremstyle{remark}

\title{On eigenmode approximation for Dirac equations:\\ differential forms and fractional Sobolev spaces}

\author{Snorre H. {\sc Christiansen}\footnote{Department of Mathematics, University of Oslo, PO box 1053 Blindern, NO-0316 Oslo, Norway. email: {\tt snorrec@math.uio.no} }}

\date{}

\begin{document}

\maketitle

\begin{abstract}
We comment on the discretization of the Dirac equation using finite element spaces of differential forms. In order to treat perturbations by low order terms, such as those arizing from electromagnetic fields, we develop some abstract discretization theory and provide estimates in fractional order Sobolev spaces for finite element systems. Eigenmode convergence is proved, as well as optimal convergence orders, assuming a flat background metric on a periodic domain.
\end{abstract}


\bigskip


\smallskip



\section{Introduction}
This paper is devoted to the discretization of the Dirac operator with finite element spaces of differential forms. We learned this technique in a talk by Ari Stern\footnote{January 2013 at the Joint Mathematics Meeting in San Diego.}, written up in \cite{LeoSte14}. It seems that a special case of this method, corresponding to tensor-product Whitney forms on cubes, has already been considered by physicists \cite{BecJoo82}. We also noticed the work \cite{LewSer10}, which, like this one, concerns spurious-free approximations to the Dirac equation, from a mathematically rigorous viewpoint. It contains many references to other works on this topic, including some from the physics literature. 

From the theoretical point of view, it should be noted that finite element methods, like the ones considered here, do not produce subspaces of the domain of the Dirac operator, not even of a square root, though $1/2$ is the critical exponent. Moreover the discrete Hodge decomposition is different from the continuous one, in the sense of Remark \ref{rem:vh}. Nevertheless we prove, in this paper, convergence for the spectral problem for finite element spaces  of differential forms $X_h$, constructed as finite element systems \cite{Chr08M3AS}, on quasi-uniform cellular complexes $\calT_h$, in the so-called $h$-setting, as the mesh-width $h$ goes to $0$. For simplicity we restrict attention to toruses, or, equivalently, periodic boundary conditions.
 
Actually, eigenmode convergence for the Hodge-Dirac operator, discretized with finite element differential forms, follows quite easily from the known theory developed for the Hodge-Laplace, which can be found in \cite{DodPat76}\cite{Chr07NM}\cite{ArnFalWin06}\cite{ChrWin08}\cite{ChrWin13IMA}. For an overview of discrete Hodge theory, see \cite{ArnFalWin10}. However, it seems that this theory by itself does not allow for zero-order perturbations of the Hodge-Dirac operator, as would correspond, for instance, to adding an electromagnetic field to the Dirac equation. Special emphasis is placed on covering such cases as well. These points are explained more fully  in Section \ref{sec:hoddir}, at the abstract level.

To use the abstract framework we develop, we require some estimates in Sobolev spaces $\rmH^s(S)$, with $0<s<1/2$, which are provided in Section \ref{sec:frac}. Notice that these Sobolev spaces contain piecewise smooth functions, in particular finite element differential forms, yet are compactly embedded in $\rmL^2(S)$. That part of the theory is developed without reference to the Dirac equation, for greater reusability. Remark \ref{rem:dischod} puts the results concerning discrete Hodge decompositions in context. As a preparation for these estimates we review some recent techniques for analysing finite element spaces of differential forms in Section \ref{sec:rev} and extend them in particular towards the use of two different discrete $\rmH^1$ seminorms.

The paper is organized as follows. We start by recalling some notations related to the Dirac equation expressed with differential forms, in Section \ref{sec:dir}. The functional framework used for the continuous  problem is reviewed in Section \ref{sec:func}. Finite element spaces are introduced in Section \ref{sec:rev} together with some error analysis in integer order Sobolev spaces. This analysis is extended to fractional order Sobolev spaces, in Section \ref{sec:frac}. We also develop some abstract discretization theory, in Section \ref{sec:hoddir}. Finally we combine these techniques to conclude, in Section \ref{sec:conc}.

\section{Forms of the Dirac equation \label{sec:dir}}

This section is meant mainly to fix notations and explain how the Dirac equation of physicists can be expressed with differential forms. For this purpose we make a detour via quaternions. We make no claims of originality for the facts listed here, but hope the exposition has other virtues.
\paragraph{Dirac equation.}
We adopt the following standard definitions:\\
The ($2 \times 2$ complex) Pauli matrices are:
\begin{equation}
\sigma_1 = 
\begin{bmatrix}
 0 & 1 \\
1 & 0
\end{bmatrix}
,\ \sigma_2 = 
\begin{bmatrix}
0 & -i \\
i & 0
\end{bmatrix}
,\ 
\sigma_3 =
\begin{bmatrix}
1 & 0 \\
0 & -1
\end{bmatrix}
.
\end{equation}
The ($4 \times 4$ complex) Dirac matrices are:
\begin{equation}
\gamma^0 =
\begin{bmatrix}
 I & 0 \\
0 & -I
\end{bmatrix}
,\
\gamma^k = 
\begin{bmatrix}
 0 & \sigma_k \\
-\sigma_k & 0
\end{bmatrix}.
\end{equation}
Here $I$ denotes the $2\times 2$ identity matrix.

The Dirac equation with minimal coupling to an electromagnetic gauge potential $A=(A_0, \ldots, A_3)$, is to find a function $\psi:\bbR^4 \to \bbC^4$ satisfying:
\begin{equation}
i \hbar \gamma^\mu (\partial_\mu + i A_\mu) \psi- mc \psi = 0. \label{eq:dir}
\end{equation}
Einstein summation convention is used. Such a function $\psi$ is referred to as a spinor-valued field on time-space. The three constants $\hbar$, $c$ and $m$ are referred to as reduced Planck constant, light velocity and mass, respectively. We choose units such that $\hbar = 1$ and $c= 1$ in the following.

\begin{remark} Concerning boundary conditions: In this section, focus is on the differential operators so we might as well work on the whole Minkowiski space $\bbR^4$ as on some open subset of it. In the next section, where Fredholmness of the Dirac operator is discussed, we recall how results on bounded contractible domains can be obtained from exactness of some sequences of differential forms with or without Dirichlet boundary conditions. For more general domains we just state the result. For the numerical method we consider, we prefer to avoid boundary conditions as much as possible, so we will work on a torus (periodic boundary conditions). Thus we avoid some interesting difficulties, but not all.
\end{remark}

\paragraph{Quaternions.}
The quaternions will be identified with the following space of complex $2\times 2$ matrices:
\begin{equation}
\bbH = \{
\begin{bmatrix}
a & - \bar b \\
b & \bar a
\end{bmatrix}
\ : \ a,b \in \bbC \}.
\end{equation}
Under matrix addition and multiplication, they constitute a real associative algebra, where nonzero elements are invertible.

The real elements in $\bbH$ are by definition the multiples of the identity matrix $I$. The imaginary elements are those spanned (over $\bbR$) by the elements $J_k = -i \sigma_k$, $k=1,2,3$. We have the direct sum decomposition:
\begin{equation}
\bbH = \bbR I \oplus \bbR J_1 \oplus \bbR J_2 \oplus \bbR J_3.
\end{equation}

For a vector $x \in \bbR^3$ we use the notation:
\begin{equation}
x \cdot J = \sum_{k=1}^3 x_k J_k.
\end{equation}
Any element $X$ of $\bbH$ can be uniquely written:
\begin{equation}
X= x_0 I + x \cdot J \quad \textrm{ with } x_0 \in \bbR \textrm{ and } x \in \bbR^3.
\end{equation}
We let $\Xi$ denote the identification: 
\begin{equation}
\Xi: \mapping{\bbR \oplus \bbR^3}{ \bbH}{(x_0, x)}{x_0 I + x \cdot J}
\end{equation}
Multiplication in $\bbH$ has the form:
\begin{equation}
(x_0 I + x \cdot J)( y_0 I + y \cdot J) = (x_0 y_0 - x \cdot y)I + (x_0 y + y_0 x + x \times y) \cdot J.
\end{equation}
Via $\Xi$ we get the following product on $\bbR \oplus \bbR^3$:
\begin{equation}
(x_0, x)(y_0, y) = (x_0 y_0 - x \cdot y, x_0 y + y_0 x + x \times y).
\end{equation}
Conjugation in $\bbH$ is Hermitian conjugation of matrices, denoted $X \mapsto X^\herm$. Via $\Xi$ it corresponds to the map:
\begin{equation}
 (x_0, x) \mapsto \overline{ (x_0, x)}=  (x_0, - x) \quad \textrm{ on } \bbR \oplus \bbR^3.
\end{equation} 
The Euclidean scalar product on $\bbH$ is:
\begin{equation}
X \cdot Y  = \Re (X^\herm Y) \label{eq:euc} = \half \tr (X^\herm Y) .
\end{equation}
Via $\Xi$ we recover the standard Euclidean product on $\bbR \oplus \bbR^3$:
\begin{equation}
(x_0, x) \cdot (y_0, y)  = \Re \overline {(x_0, x)} (y_0, y) = x_0 y_0 + x \cdot y.
\end{equation}

\paragraph{Differential operators.} The differential operator $J \cdot \nabla$ is by definition:
\begin{equation}
J \cdot \nabla = \sum_{k=1}^3 J_k \partial_k.
\end{equation}
It acts in particular on fields $\bbR^3 \to \bbH$ and as such is symmetric, with respect to the $\rmL^2$ product deduced from the Euclidean product (\ref{eq:euc}). On the other hand, the operator $ \sigma \cdot \nabla$, defined in a similar manner, acts from fields $\bbR^3 \to \bbH$ to  fields $\bbR^3 \to i \bbH$.
We notice that we have a direct sum decomposition (over the reals):
\begin{equation}
\bbC^{2 \times 2} = \bbH \oplus i\bbH.
\end{equation}
Moreover $ \sigma \cdot \nabla$, as an operator on fields $\bbR^3 \to \bbC^2$,  is anti-symmetric.

We change the notation slightly, and write the electromagnetic gauge potential as  combination of an electric scalar potential $V$ (corresponding to $A_0$) and an $\bbR^3$-valued magnetic vector potential $A=(A_1, A_2, A_3)$.
In terms of $2\times 2$ matrices the Dirac equation (\ref{eq:dir}), acting on $\bbC^4$-valued fields, can be written:
\begin{equation}
 i \begin{bmatrix}
 \partial_0 +i V & 0 \\
0 & -\partial_0 - iV
\end{bmatrix} \psi = -  \begin{bmatrix}
 0 & - J \cdot (\nabla  + i A)\\
J \cdot (\nabla + i A) & 0
\end{bmatrix} \psi + m \psi.
\end{equation}
We can also write this as:
\begin{equation}
 \partial_0 \psi  = 
 -  \begin{bmatrix}
 0 & \sigma \cdot (\nabla  + i A)\\
\sigma \cdot (\nabla + i A) & 0
\end{bmatrix} \psi 
- mi \begin{bmatrix}
I & 0 \\
0 & -I
\end{bmatrix}
   \psi  -i V\psi.
\end{equation}
The differential operator appearing on the right hand side is anti-selfadjoint, in $\rmL^2(\bbR^3 \to \bbC^4)$, so the differential equation may be regarded as a well-posed equation of wave type.

If we look for solutions with a time-dependence of the form $t \mapsto \exp(-iEt)$ with $E \in \bbR$, we get the self-adjoint eigenvalue problem:
\begin{equation}
E \psi = \begin{bmatrix}
 0 & J \cdot (\nabla  + i A)\\
J \cdot (\nabla + i A) & 0
\end{bmatrix} \psi + 
m
\begin{bmatrix}
I & 0 \\
0 & -I
\end{bmatrix}
\psi + V \psi. \label{eq:saeig}
\end{equation}

\paragraph{Differential forms.}
Via $\Xi$, the operator $J \cdot \nabla$ becomes the following operator on fields $\bbR^3 \to \bbR \oplus \bbR^3$:
\begin{equation}
 \Xi^{-1} (J \cdot \nabla) \Xi: 
\begin{bmatrix}
f \\
g
\end{bmatrix}
\mapsto
\begin{bmatrix}
- \div g\\
\grad f + \curl g
\end{bmatrix}.
\end{equation}

Recall the identification of alternating forms with scalars and vectors:
\begin{equation}
\alter^\bs (\bbR^3) \approx  \bbR \oplus \bbR^3 \oplus \bbR^3 \oplus \bbR.
\end{equation}
Introduce the rearrangement map:
\begin{equation}
\Theta : \begin{bmatrix}
\bbR \\
\bbR^3 \\
\bbR^3 \\
\bbR
\end{bmatrix}
\ni
\begin{bmatrix}
 s\\
 u\\
 v\\
 t
\end{bmatrix}
\mapsto
\begin{bmatrix}
s\\
v\\
t\\
u
\end{bmatrix}
\in
\begin{bmatrix}
\bbR \\
\bbR^3 \\
\bbR \\
\bbR^3
\end{bmatrix}.
\end{equation}
It can be interpreted as putting the alternating forms  of even order on top and those of odd order at the bottom.

Looking at the right hand side in (\ref{eq:saeig}), we define the differential operator:
\begin{equation}
\calD = 
\Theta^{-1}
\begin{bmatrix}
0 &  \Xi^{-1} (J \cdot \nabla) \Xi \\
 \Xi^{-1} (J \cdot \nabla) \Xi & 0 
\end{bmatrix}
\Theta.
\end{equation}
We obtain the mapping property, on scalar and vector fields:
\begin{equation} \calD : 
\begin{bmatrix}
s\\
u\\
v\\
t
\end{bmatrix}
\mapsto
\begin{bmatrix}
- \div u\\
\grad s + \curl v\\
\curl u + \grad t\\
-\div v
\end{bmatrix}.
\end{equation}
This map is of the form:
\begin{equation}
\calD = \rmd + \rmd^\star: C^\infty \alter^\bs (\bbR^3) \to C^\infty \alter^\bs (\bbR^3),
\end{equation}
for the slightly non-standard exterior derivative:
\begin{equation}
\rmd = \grad \oplus \curl \oplus -\div .
\end{equation}

Therefore $\calD$ may be referred to as a Hodge-Dirac operator.

\paragraph{Almost complex structure.} Unfortunately the magnetic vector potential does not contribute nicely under these identifications. This is related to the fact that, for $A \in \bbR^3$,   $J \cdot i A$ maps $\bbH$ not to itself but to $i \bbH$ (compare with the fact the $J \cdot \nabla$ acts as an endomorphism on smooth fields $S \to \bbH$). In other words, under the above identifications, the magnetic vector potential yields a zero order operator on  $C^\infty \alter^\bs (\bbR^3)$, but we found the expression to be unenlightening.

The following identities will not be used, but we found them intriguing. At least they provide an interpretation of the fact that the spectrum of the Dirac operator is symmetric with respect to the origin. Introduce the operator on $\alter^\bs (\bbR^3) $:
\begin{equation}
\calJ :
\begin{bmatrix}
s\\
u\\
v\\
t
\end{bmatrix}
\mapsto
\begin{bmatrix}
-t\\
v\\
-u\\
s
\end{bmatrix}.
\end{equation}
Since $\calJ^2 = -1$, $\calJ$ provides an almost complex structure on  $\alter^\bs(\bbR^3)$. Up to signs one can also identify it as a Hodge-star operator. We notice:
\begin{equation}
\calD\calJ = - \calJ \calD.\label{eq:anticom}
\end{equation}
In particular, if $\phi $ is an eigenvector of  $\calD$ with eigenvalue $\lambda$,  then $\calJ \phi$ is an eigenvector for the eigenvalue $-\lambda$.

One way of viewing this is to introduce the map:
\begin{equation}
\calZ \ : \ \mapping{\alter^\bs(\bbR^3) }{\bbC^{2 \times 2}}{(s,u,v,t)}{(s + v \cdot J) + i(t + u\cdot J)}
\end{equation}
We also let $\calH$ denote Hermitian conjugation on $\bbC^{2 \times 2}$. Some computations then provide the formulas:
\begin{equation}
\calZ \calD \calZ^{-1} = (\sigma \cdot \nabla) \circ \calH,
\end{equation}
and:
\begin{equation}
\calZ \calJ \calZ^{-1} = i.
\end{equation}
From this, identity (\ref{eq:anticom}) follows, since $\sigma\cdot \nabla$ is $\bbC$-linear, whereas $\calH$ is $\bbC$-antilinear.

\section{Functional framework for the Hodge-Dirac \label{sec:func}}

We fix some notations and recall some rather well-known results. The Hilbertian setting for the exterior derivative is in \cite{ArnFalWin10}. Additional information on the Dirac operator can be found in \cite{Tay96II}.

\paragraph{Contractible domains.} Let $S$ be a bounded contractible Lipschitz domain in $\bbR^n$, with its standard Euclidean structure. We study the Hodge-Dirac operator on the space of differential forms: 
\begin{equation}
\calD = \rmd + \rmd^\star, \textrm{ on  } O^\bs = \rmL^2 \alter^\bs(S).
\end{equation}
It is an unbounded operator and we proceed to define its domain. The grading implied by the notation $O^\bs$ is given by the degree of differential forms and the scalar product in $O^\bs$ is the $\rmL^2$ product of forms deduced from the standard Euclidean structure on $\bbR^n$.

We denote by $\gamma$ the pullback of differential forms to the boundary $\partial S$  of $S$. If $u$ is a differential form, $\gamma u$ remembers the action of $u$ only on vectors tangent to $\partial S$. We put, for integer $k$:
\begin{align}
X^k & = \{ u \in O^k \ : \rmd u \in O^{k+1}\},\\
X^k_0 & = \{ u \in X^k \ : \  \gamma u = 0 \},\\
Y^k & = \{ u \in X^k_0 \ : \ \rmd^\star u \in O^{k- 1}\}.
\end{align}
When $k$ is outside the range $[0, n]$, these spaces are $0$. In these formulas, the exterior derivative $\rmd$ is defined a priori in the sense of distributions, and $\rmd^\star$ is its formal adjoint, also defined in the sense of distributions. We define the domain of $\calD$ to be $Y^\bs$. Sometimes, to simplify notations we omit the grading, so that $O = O^\bs$, $X = X^\bs$, $X_0 = X^\bs_0$ and $Y = Y^\bs$.

It is straightforward to check that $\calD$ with domain $Y$ is symmetric. Self-adjointness follows from the next results.

Recall the following two variants of the Poincar\'e lemma:
\begin{proposition} The following two sequences, involving the exterior derivative, are exact:
\begin{equation}
0 \to \bbR \to X^0 \to \ldots \to X^n \to 0,
\end{equation}
with inclusion of constants in second position, and:
\begin{equation}
0 \to X^0_0 \to \ldots \to X^n_0 \to \bbR \to 0,
\end{equation}
with integration in second to last position.
\end{proposition}
One deduces the following variant of the Hodge decomposition. Notice in particular the choice of boundary conditions.
\begin{proposition}
For any $v \in O^k$, with $\int v = 0$ if $k=n$, there is a unique $(u,w)\in Y^{k-1} \times Y^{k+1}$ such that:
\begin{equation}
v = \rmd u + \rmd^\star w,
\end{equation}
subject to $\int u = 0$ if $k-1 = 0$ and $\rmd^\star u = 0$  if not, as well as   $\int w = 0$ if $k+1 = n$ and $\rmd w = 0$ if not.
\end{proposition}
It is the main ingredient in the proof of the following:
\begin{proposition}
For any $ v^\bs \in O^\bs$, with $\int v^n = 0$,  there is a unique $u^\bs \in Y^\bs$, with $\int u^n = 0$, such that:
\begin{equation}
v^\bs = \calD u^\bs.
\end{equation}
\end{proposition}

\paragraph{More general domains.} For more general compact Lipschitz domains $S$, including compact smooth manifolds without boundary, the inclusion map $Y \to O$ is still compact and the operator $\calD : Y \to O$ is still Fredholm of index zero. The kernel  of $\calD$ is the graded cohomology group:
\begin{equation}
G = \{u \in X_0 \ : \ \rmd u = 0 \textrm{ and } u \perp \rmd X_0 \} \label{eq:g}.
\end{equation}
We also adopt the following notations:
\begin{align}
W & = \{u \in X_0 \ : \ \rmd u = 0 \}, \label{eq:w}\\
V & = \{ u \in X_0 \ : \ u \perp W\} \label{eq:v}.
\end{align}
We have the Hodge decompositions:
\begin{align}
X_0 & = V \oplus W, \textrm{ and } W  = \rmd V \oplus G.
\end{align}
Define then an operator $\calK: O \to Y$ by inverting $\calD$ on the $O$-orthogonal of $G$. We note, for future reference:
\begin{proposition}
The following operator is compact and selfadjoint:
\begin{equation}
\calK: O \to O.
\end{equation}
\end{proposition}

\section{Finite element differential forms \label{sec:rev}}

\paragraph{Discretization spaces.}

In this section we choose the domain $S$ to be of the following type. We let $(e_i)$ denote a basis of $\bbR^n$ ($1\leq i \leq n$) and define a lattice $\bbL$ by:
\begin{equation}
\bbL= \sum_{i=1}^n \bbZ e_i.
\end{equation}
Then we define:
\begin{equation}
S = \bbR^n /\bbL.
\end{equation}

We also consider a quasi-uniform sequence $(\calT_h)$ of cellular complexes. As is customary, the parameter $h>0$ also denotes the largest diameter of a cell in $\calT_h$, and we are interested in the asymptotic behavior as $h \to 0$. Quasi-uniformity can be taken to mean that $\calT_h$ has a simplicial refinement, which is quasi-uniform in the usual sense, and such that the cells of $\calT_h$ are composed of a uniformly bounded number of simplices.

\begin{remark}
Both the choice of working on a domain without boundary and the choice to restrict attention to quasi-uniform meshes, are made in order to use smoothing by convolution, in combination with standard interpolation, in the proofs. It is possible that the techniques of \cite{ChrWin08} (which introduces a space dependent smoother and an offset to treat essential boundary conditions) can be extended to treat the Dirac operator on domains with boundary and general shape-regular meshes, but the proofs are already quite technical without these complicating factors.
\end{remark}

Recall that $| \cdot |$ stands for the $\rmL^2(S)$ norm and $\langle \cdot , \cdot \rangle$ for the $\rmL^2(S)$ scalar product. We also use the notation $O = O^\bs$ for the Hilbert space of differential forms of all degrees, equipped with this scalar product. For any cell $T$ of some $\calT_h$, also the low-dimensional ones, we write:
\begin{equation}
|u|_{T}= \|u\|_{\rmL^2(T)}. 
\end{equation}

For each discretization parameter $h$, we consider a sequence of finite dimensional spaces $X^k_h$ of $k$-forms ($0 \leq k \leq n$). The exterior derivative should induce maps $\rmd: X^k_h \to X^{k+1}_h$. We define the graded space $X_h$:
\begin{equation}
X_h = \bigoplus_{k} X^k_h.
\end{equation}
The graded space $X_h$ is equipped with the endomorphism $\rmd$ and the scalar product deduced from the $\rmL^2$ product of differential forms. We denote by $W_h$ the kernel of $\rmd$ on $X_h$ and $V_h$ the orthogonal of $W_h$ in $X_h$, with respect to the $\rmL^2$ inner product. That is, comparing with (\ref{eq:w},\ref{eq:v}), we put:
\begin{align}
W_h & = \{u \in X_h \ : \ \rmd u = 0 \},\\
V_h & = \{ u \in X_h \ : \ u \perp W_h\}.
\end{align}
\begin{remark}\label{rem:vh}
A key point is that we are interested in cases where $V_h$ is not a subspace of $V$.
\end{remark}
Comparing with (\ref{eq:g}), we may also identify the graded cohomology group:
\begin{equation}
G_h = \{ u \in X_h \ : \ \rmd u = 0 \textrm{ and } u \perp \rmd X_h \}.
\end{equation}

We suppose that the spaces $X_h$ have been constructed as finite element systems on $\calT_h$, in the $h$-setting. The framework of finite element systems was developed in particular in \cite{Chr08M3AS}\cite{Chr09AWM}\cite{ChrMunOwr11}. It includes the finite element spaces of differential forms treated in \cite{Hip02}\cite{ArnFalWin06}. See \cite{ChrRap16} for more information on how standard mixed finite elements fit in the framework of finite element systems. This framework also includes some more recent finite elements developed in \cite{Chr10CRAS}\cite{ArnAwa14}\cite{CocQiu14}. These are all minimal finite element systems under various constraints, as detailed in \cite{ChrGil16}.

\paragraph{Estimates for smoothed projections.}

Estimates of the form, there exists a constant $C \geq 0$ such that for all $h$,
\begin{equation}
A_h \leq C B_h,
\end{equation}
where $A_h$ and $B_h$ are some $h$-dependent quantities (typically some norm of some elements of $X_h$) will be written:
\begin{equation}
A_h \cleq B_h.
\end{equation}

  Stability estimates for discrete Hodge decompositions have the following form:
\begin{proposition}\label{prop:basic}
(i) Choose a $k$-form $u^k\in X_h^k$. We write the discrete Hodge decomposition $u^k = \rmd v^{k-1} + v^k + g^k$, with $v^{k-1} \in V^{k-1}_h$, $v^k \in V^k_h$ and $g^k \in G_h^k$. Then we have estimates:
\begin{align}
|v^{k-1}| & \cleq |u^k|. \label{eq:vkm}
\end{align}
(ii) Continuous and discrete cohomology groups are related by:
\begin{equation}
\hat \delta(G_h, G) \to 0,
\end{equation}
where the symmetrized gap $\hat \delta$ is evaluated in the $O$-norm.\\
(iii) For any subsequence $v_h \in V_h$ with $|\rmd v_h|$ bounded, there is a subsubsequence converging in $O$ to an element in $V$.
\end{proposition}

Recall that the gap and symmetrized gap between to subspaces $A_h$ and $B_h$ of a normed vector space, are defined as:
\begin{align}
\delta(A_h, B_h) & = \sup_{u \in A_h \, : \, \| u \| = 1} \inf_{v \in B_h} \| u - v \|,\\
\hat \delta (A_h, B_h) & = \max \{ \delta(A_h, B_h), \delta(B_h,A_h) \}.
\end{align}

Proofs of the above three statements (i), (ii) and (iii) were provided for closed manifolds in Proposition 9, Corollary 4 and Corollary 5 of \cite{Chr07NM}. That $G_h$ and $G$ have the same dimension is essentially de Rham's theorem. That the symmetrized gap between $G_h$ and $G$ goes to $0$ can be interpreted as an example of eigenmode convergence. Convergence of other eigenmodes usually relies on the third property, referred to as discrete compactness. 

Discrete compactness properties, as introduced by Kikuchi \cite{Kik89} for Maxwell eigenmode problems in cavities, have been thoroughly studied in finite element contexts \cite{Bof10}. We refer to \cite{ChrWin13IMA} for more details on eigenmode convergence at an abstract level, exploring both sufficiency and necessity of various conditions on discretization spaces. A first proof of eigenmode convergence for the Hodge-Laplace,  discretized with Whitney forms, was provided in \cite{DodPat76}. See \cite{ArnFalWin10} for a more detailed overview on continuous and discrete Hodge theory.

\begin{remark}\label{rem:ggh}
When $S$ is a flat torus, the kernel $G$ consists of the constant differential forms, and so does $G_h$ provided that they are contained in the Galerkin space, so that $\hat \delta(G_h, G) = 0$. However we prefer to refrain as much as possible from appealing to this special property of toruses in our discussions.
\end{remark}

The most efficient tool so far to prove discrete compactness and related estimates, seems to be a combination of standard interpolation with a smoothing technique, a method that has been developed in  \cite{DodPat76}\cite{Sch08}\cite{Chr07NM}\cite{ArnFalWin06}\cite{ChrWin08}\cite{ChrMunOwr11}. In this paper we use the simplest smoothing technique, consisting of smoothing by convolution.

Let then $\phi$ be a mollifier: a smooth function on $\bbR^n$ which is supported in the unit ball and positive (in the French sense, which is more permissive), with integral $1$. For $\epsilon >0$ we define $\phi_\epsilon$ on $\bbR^n$ by:
\begin{equation}
\phi_\epsilon(x) = \epsilon^{-n} \phi(\epsilon^{-1} x),
\end{equation}
so that $\phi_\epsilon$ is supported in the ball with radius $\epsilon$ and has integral $1$.

We let $I_h$ be a standard interpolator onto $X_h$, defined for smooth functions and commuting with the exterior derivative. Typically $I_h$ is the projection deduced from a choice of degrees of freedom and is ill-defined on $O$.

We recall the main steps in the construction of mollified interpolators $\Pi_h: O \to X_h$.  They can also be referred to as smoothed projections.
\begin{proposition} \label{prop:l2stab}
 (i)  For any $\delta >0$, by fixing $\epsilon$ small enough, we have for all $h$ and all $u \in X_h$:
\begin{align}
| u - \phi_{\epsilon h} \ast u | \leq \delta |u |,\\
| u - I_h \phi_{\epsilon h} \ast u | \leq \delta |u |.
\end{align}
In particular, for $\delta <1$ we get a norm-equivalence on $X_h$:
\begin{equation}
|u| \ceq | \phi_{\epsilon h} \ast u | \ceq  | I_h (\phi_{\epsilon h} \ast u) |,
\end{equation}
(ii) For small enough fixed $\epsilon>0$, the operators:
\begin{equation}\label{eq:pio}
\mapping{O}{X_h}{u}{I_h \circ (\phi_{\epsilon h} \ast u)} 
 \end{equation}
are $h$-uniformly stable $O \to O$.\newline
(iii) For a small enough fixed $\epsilon >0$, composing (\ref{eq:pio}) with the inverse of its restriction to $X_h$, we get a projector $\Pi_h: O \to X_h$ which is uniformly stable  $O\to O$  and commutes with the exterior derivative on $X$.
\end{proposition}

The $\epsilon$ appearing in this proposition is chosen so that in particular the $(\epsilon h)$-neighborhood of any cell $T \in \calT_h$ is included in the macroelement consisting of cells in $\calT_h$ touching $T$. The constants that are implicit in the estimates will typically depend on the constants that determine the shape-regularity and quasi-uniformity of the family of meshes.

One way to sum up the virtues of the smoothed projections is to say that commutation with the exterior derivative guarantees that $W_h$ is a nice subspace of $W$, whereas stability in $O$ guarantees that $V_h$ is close to $V$. These themes will be more amply developed below, in particular by giving several precise interpretations of the second principle.

In all of the following, $p$ is a natural number, which could very well be $0$. We suppose that the finite element spaces contain the differential forms which are polynomials of degree up to $p$. We choose the mollifier $\phi$ so that for any polynomial $u$ of degree up to $p$:
\begin{equation}
\int \phi(-x)u(x) \rmd x = u(0).
\end{equation}
Then convolution by $\phi_{\epsilon h}$ preserves polynomials of degree up to $p$. 

We get the following optimal orders of convergence:
\begin{proposition} We have:
\begin{equation}
| u -\Pi_h u | \cleq | u - I_h (\phi_{\epsilon h} \ast u) |   \cleq h^{p+1} | \nabla^{p+1} u|.
\end{equation}
\end{proposition}

\paragraph{A discrete $\rmH^1$ semi-norm.}
The following discrete semi-norn is defined for fields $u$ which are piecewise of class $\rmH^1$ with respect to $\calT_h$:
\begin{equation}
\lceil u \rceil ^2_h= \sum_{T} | \nabla u |^2_T + \sum_{T'} h_T^{-1} | \ljump  u \rjump |_{T'}^2.
\end{equation}
On the right hand side, the first term is the broken $\rmH^1(S)$ semi-norm, and the second term collects jumps on codimension one interfaces, with an appropriate scaling.

Remark that $\rmd^\star u$ is not well defined in $\rmL^2(S)$ for $u \in X_h$. However it may be defined weakly by integration by parts and we then have a bound on the norm, as follows.

\begin{proposition} \label{prop:dshonesemi}
Suppose $u$ is piecewise $\rmH^1$ with respect to the mesh $\calT_h$ (for instance $u \in \rmH^1\alter^\bs(S) + X_h$). We have an estimate:
\begin{equation}
\sup_{v \in X_h} \frac{|\langle u , \rmd v\rangle|}{|v|} \cleq \lceil u \rceil_h.
\end{equation}
\end{proposition}
\begin{proof}
With the above notations, and introducing the Hodge star operator $\hs$, we write:
\begin{equation}
\langle u,  \rmd v \rangle = \sum_T \int_T u \cdot \rmd v =  \pm \sum_T \int_T \hs u \wedge \rmd v.
\end{equation}
Then we integrate each term on the right by parts, in the cell $T$. Collecting terms, we identify jumps of $\hs u$ on the interfaces between cells.

Next we use a discrete trace theorem for elements of $X_h$, obtained by scaling: when $T'$ is a codimension-1 face of a cell $T''$ we have an estimate:
\begin{equation}
\| v \|_{\rmL^2(T')} \cleq h_{T''}^{-1/2} \| v \|_{\rmL^2(T'')}.
\end{equation}

All in all we get:
\begin{equation}
|\langle u,  \rmd v \rangle | \cleq  \sum_T  |\int_T \rmd^\star u \cdot v  |+  \sum_{T'} h_{T''}^{-1/2} | \ljump \hs u \rjump  |_{T'}  | v |_{T''} ,
\end{equation}
where for each interface $T'$ we choose a cell $T''$ containing it.

One concludes with Cauchy-Schwarz.
\end{proof}

Recall the following result from \cite{ChrMunOwr11} Section 5.4, Proposition 5.67 (notice changes in notations compared with that paper):
\begin{proposition} \label{prop:oneequiv} 
Fix $\epsilon>0$ small enough. For all $h$ we have, for all $u \in X_h$, an equivalence of discrete seminorms:
\begin{equation}
| \nabla (\phi_{\epsilon h} \ast u) | \approx \lceil  u \rceil_h .
\end{equation}
\end{proposition}

In what follows, recall that $\Pi_h$ denotes the smoothed projection (mollified interpolator) defined in Proposition \ref{prop:l2stab}.

For the following we refer to Proposition 5.68 of \cite{ChrMunOwr11}, which extends Proposition \ref{prop:l2stab} from $\rmL^2$ to a discrete $\rmH^1$ setting. It is proved similarly, by scaling from reference elements.
\begin{proposition} \label{prop:smstabh1} For any $\delta >0$, for fixed $\epsilon$ small enough, we have for $u \in X_h$:
\begin{equation}
\lceil u - I_h (\phi_{\epsilon h} \ast u)  \rceil_h  \leq \delta \lceil \nabla u\rceil_h.
\end{equation}
For $u\in \rmH^1\alter^\bs(S)$ we have estimates:
\begin{align}
\lceil I_h (\phi_{\epsilon h} \ast u)  \rceil_h & \cleq | \nabla u|,\\
\lceil \Pi_h u \rceil_h & \cleq | \nabla u|.
\end{align}
\end{proposition}

\paragraph{Estimates on discrete Hodge decompositions.}

Recall the notations (\ref{eq:w},\ref{eq:v}). We let $\overline V$ be the completion of $V$ in $O$. We also let $\frH$ denote the $O$-orthogonal projection onto $\overline V$. It realizes a Hodge decomposition in the form $u = \frH u + (u - \frH u)$. Remark that, since $u - \frH u \in W$, we have  $\rmd \frH u = \rmd u$.

Discrete compactness properties are usually deduced from the following inequality (e.g. \cite{ChrMunOwr11} Proposition 5.66):
\begin{lemma}\label{lem:mag}
 For $u\in V_h$:
\begin{equation}
| u - \frH u| \leq | \frH u - \Pi_h \frH u|. 
\end{equation}
\end{lemma}
One deduces for instance:
\begin{proposition} \label{prop:gap} For $u \in V_h$:
\begin{equation} 
|u - \frH u | \leq h |\rmd u|.
\end{equation}
\end{proposition}
In this proposition, the power of $h$ will be lower in cases where full elliptic regularity does not hold. This estimate may be interpreted as a convergence of the gap $\delta(V_h, V)$ to $0$, in the $X$ norm. In \cite{ChrWin13IMA} we showed that this condition is intermediate between two discrete Friedrich estimates, equivalent to eigenmode convergence and related to estimates on bounded commuting projections, all at an abstract level, but for semi-definite problems. See Remark \ref{rem:dischod} below.

Define a projector $P_h : X \to V_h$, by imposing, for all $v\in V_h$:
\begin{equation}\label{eq:phvh}
\langle \rmd P_h u, \rmd v \rangle = \langle \rmd u, \rmd v \rangle .
\end{equation}

The following is an Aubin-Nitsche trick for $P_h$. Compared with the standard one, the main difficulty is that $V_h$ is not a subspace of $V$. As always, the discrepancy is handled with Proposition \ref{prop:gap}. The point is to obtain estimates in $O$ for a variational problem which is naturally wellposed in $V$, as in \S 3.5 of \cite{ArnFalWin10}.

\begin{proposition}\label{prop:phan}
For $u \in V$ we have an error estimate:
\begin{equation}
| u -P_h u| \cleq h |\rmd u|.
\end{equation}
\end{proposition}
\begin{proof}
We will use elliptic regularity on the torus, in the form:
\begin{equation}
\inf_{u \in V} \sup_{v \in V \cap \rmH^2(S)} \frac{\langle \rmd u ,\rmd v \rangle}{| u | \, \|v\|_{\rmH^2}} \cgeq 1.
\end{equation}

Choose now $u\in V$. We have, using in particular Proposition \ref{prop:gap}:
\begin{align}
| u - P_h u| & \leq | u - \frH P_h u| + | P_h u - \frH P_h u |, \\
&\cleq  | u - \frH P_h u| + h | \rmd P_h u |, \\
&\cleq  | u - \frH P_h u| + h | \rmd u |. \label{eq:alm}
\end{align}
Since $ u - \frH P_h u \in V$ we may write:
\begin{align}
| u - \frH P_h u| & \cleq \sup_{v \in V \cap \rmH^2(S)}\frac{\langle \rmd (u  - \frH P_h u),\rmd v \rangle}{\|v\|_{\rmH^2}}, \\
& \cleq \sup_{v \in V \cap \rmH^2(S) }\frac{\langle \rmd (u  -  P_h u),\rmd (v - v_h) \rangle}{\|v\|_{\rmH^2}}, 
\end{align}
for any choice $v_h \in V_h$. Let $Q_h$ be the $O$-projection onto $W_h$. We choose $v_h = \Pi_h v - Q_h \Pi_h v \in V_h$ and remark that $\rmd v_h = \rmd \Pi_h v$. Therefore:
\begin{align}
| u - \frH P_h u| & \cleq  \sup_{v \in V \cap \rmH^2(S) }\frac{\langle \rmd (u  -  P_h u),\rmd (v - \Pi_h v) \rangle }{\|v\|_{\rmH^2}}, \\
& \cleq  \sup_{v \in V \cap \rmH^2(S) }\frac{| \rmd (u  -  P_h u)| \, | \rmd (v - \Pi_h v)| }{\|v\|_{\rmH^2}}. 
\end{align}
Since $ |\rmd P_h u | \cleq |\rmd u|$ and moreover:
\begin{align}
 | \rmd (v - \Pi_h v)|  = | \rmd v - \Pi_h \rmd v| \cleq h \|\rmd v\|_{\rmH^1} \cleq h \|v\|_{\rmH^2},
\end{align}
we get:
\begin{equation}
 | u - \frH P_h u| \cleq h |\rmd  u|.
 \end{equation}
Plugging this into (\ref{eq:alm}) we conclude. 
\end{proof}

\paragraph{A discrete domain semi-norm.}

In this paragraph we introduce a discrete semi-norm motivated by the domain of the Hodge-Dirac operator. One first checks that for $u \in \rmL^2\alter^\bs(S)$, $\calD u \in \rmL^2(S)$ if and only if $\rmd u \in \rmL^2$ and $\rmd^\star u \in \rmL^2$. On flat toruses one may then integrate by parts the identity:
\begin{equation}
\nabla^\star \nabla = \Delta = \rmd^\star \rmd + \rmd \rmd^\star.
\end{equation}
 One gets that $\calD u \in \rmL^2(S)$ if and only if $u\in \rmH^1(S)$. A discrete analogue of this result will play a role later on.

We define the discrete domain semi-norm:
\begin{equation}\label{eq:ysem}
[ u ]^2_h =   | \rmd u |^2 + ( \sup_{v \in X_h} \frac{ |\langle u,  \rmd v\rangle | }{|v|})^2.
\end{equation}
On $X_h \cap G_h^\perp$ it defines a norm which dominates the $\rmL^2$ norm, as follows from Proposition \ref{prop:basic}.

\begin{proposition}\label{prop:honedom}
We have a uniform equivalence of norms, for $u \in X_h$:
\begin{equation}
[ u ]_h \ceq \lceil u \rceil_h.
\end{equation}
\end{proposition}
\begin{proof}
It follows from Proposition \ref{prop:dshonesemi} that we have a bound:
\begin{equation}
[ u ]_h \cleq \lceil u \rceil_h.
\end{equation}
We proceed to prove the converse bound.

We write, using first Proposition \ref{prop:oneequiv}, then integrating the identiy $\Delta = \rmd^\star \rmd + \rmd \rmd^\star$ by parts, and finally using Proposition \ref{prop:l2stab}:  
\begin{align}
\lceil u \rceil_h & \ceq | \nabla \phi_{\epsilon h } \ast u |,\\
& \ceq  | \rmd (\phi_{\epsilon h } \ast u) | +  | \rmd^\star (\phi_{\epsilon h } \ast u) |,\\
& \ceq  | \rmd u | +  | \rmd^\star (\phi_{\epsilon h } \ast u) |, \label{eq:dsll}
\end{align}
for $\epsilon$ small enough.

To estimate the last term in (\ref{eq:dsll}) we do the discrete Hodge decomposition of $u$:
\begin{equation}
u =  v +  \rmd w + g, \textrm{ with } v \in V_h,\ w\in W_h \textrm{ and } g \in G_h.
\end{equation}
We treat $v$, $w$ and $g$ separately.

(i) We first treat $v$. Convolution by $\phi_{\epsilon h}$ commutes with the exterior derivative, so preserves $W$. Since it is also selfadjoint it preserves the orthogonal $\overline V$. And since it improves regularity it also preserves $V$.
\begin{align}
| \rmd^\star (\phi_{\epsilon h} \ast v) | & = | \rmd^\star (\phi_{\epsilon h} \ast v  - \phi_{\epsilon h} \ast \frH v) |,\\
& \cleq | \nabla \phi_{\epsilon h} \ast (v  -\frH v)|,\\
& \cleq h^{-1} |v  -\frH v|,\\
& \cleq | \rmd v|.
\end{align}

(ii) For $w$ we write:
\begin{align}
 | \rmd^\star (\phi_{\epsilon h } \ast \rmd w) | & \cleq \sup_{w' \in V} \frac{| \langle \phi_{\epsilon h} \ast \rmd w, \rmd w'\rangle  |}{|w'|},\\
& \cleq \sup_{w' \in V} \frac{| \langle \rmd w, \rmd \phi_{\epsilon h} \ast  w'\rangle  |}{|w'|},\\
& \cleq \sup_{w' \in V} \frac{| \langle \rmd w, \rmd P_h (\phi_{\epsilon h} \ast  w')\rangle  |}{|w'|}.
\end{align} 
Then we write, for $w' \in V$, since $ \phi_{\epsilon h} \ast  w' \in V$:
\begin{align}
| P_h (\phi_{\epsilon h} \ast  w' )| & \leq | \phi_{\epsilon h} \ast  w'  | + | \phi_{\epsilon h} \ast  w' - P_h (\phi_{\epsilon h} \ast  w') |,\\
& \cleq | w' | + h |  \rmd (\phi_{\epsilon h} \ast  w') |, \\
& \cleq | w' | + h |  \rmd   w' |, \\
& \cleq | w' |.
\end{align}
Therefore:
\begin{align}
| \rmd^\star (\phi_{\epsilon h } \ast \rmd w) |   \cleq \sup_{w' \in V_h} \frac{| \langle \rmd w, \rmd  w'\rangle  |}{|w'|}.
\end{align}

(iii) For $g$ we skip a general proof. On a flat torus the constants are stable under convolution and have coderivative zero, so $| \rmd^\star (\phi_{\epsilon h } \ast g) | = 0$.
\end{proof}

\section{Fractional order Sobolev space estimates \label{sec:frac}}

In the following, we obtain estimates in fractional order Sobolev spaces, for finite element systems, using the technique of mollified interpolators. These results include three inverse estimates, stability of mollified interpolators, approximations orders and stability of discrete Hodge decompositions, all with respect to fractional order Sobolev norms. 
 
Recall the definition of the Slobodetskij semi-norms on $\rmH^s(S)$ for $s \in ]0,1[$:
\begin{equation}
\lfloor u \rfloor_{s,S}^2 = \iint_{S \times S} \frac{|u(x) - u(y)|^2}{| x - y |^{n + 2s} }\rmd x \rmd y.
\end{equation}
In this notation we may occasionally replace $S$ by a subdomain of $S$. We often omit $S$ from the notation when the whole domain is considered. Such semi-norms have already been extensively used in finite element contexts, see for instance \cite{BreSco08} (Chapter 14). See also \cite{Heu14} for details on interpolation inequalities, in particular how the constants may be taken uniform with respect to a family of domains, such as a compact family of reference macro-elements of given combinatorial structure. The interested reader is also referred to \cite{HipLiZou12}.

\paragraph{Stability and inverse inequalities.}

\begin{proposition}\label{prop:deleps}
For any $0 < s < 1/2$ and any $\delta>0$, fixing $\epsilon$ small enough, we have for all $h$ and all $u \in X_h$:
\begin{equation}
|u - \phi_{\epsilon h} \ast u | \leq \delta h^s \lfloor u \rfloor_s .
\end{equation}
\end{proposition}
\begin{proof}
On a cell $T$ of diameter $1$, with a corresponding macroelement $\calM(T)$, we may get an estimate:
\begin{equation}
|u - \phi_{\epsilon h} \ast u |_{T} \leq \delta \lfloor u \rfloor_{s, \calM (T)} .
\end{equation}
Then we scale to size $h$, square and sum over $T$. Then we bound the sum of semi-norms squared on the right hand side, by the semi-norm squared on the whole domain.
\end{proof}

The following is an inverse inequality for the Slobodetskij semi-norms.
\begin{proposition} \label{prop:slobinv}
For $0<s<s' <1/2$, we have $h$-uniform estimates for $u \in X_h$:
\begin{equation}
\lfloor u \rfloor_{s'} \cleq h^{s-s'} \lfloor u \rfloor_{s} .
\end{equation}
\end{proposition}
\begin{proof}
We square the two sides of the inequality to be proved, and write the semi-norm over $S$ as a sum indexed by pairs of cells $T$ and $T'$ in $\calT_h$. That is:
\begin{equation}
\lfloor u \rfloor_{s}^2 = \sum_{T, T'} \iint_{T \times T'} \frac{|u(x) - u(y)|^2}{| x - y |^{n+ 2s } }\rmd x \rmd y.
\end{equation}
We distinguish two kinds of configurations for $T, T' \in \calT_h$:

(i) If $T \cap T' \neq \emptyset$, we use the Slobodetskij semi-norm on $T \cup T'$. We use scaling from an estimate on reference elements, to get:
\begin{equation}
\iint_{T \times T'} \frac{|u(x) - u(y)|^2}{| x - y |^{n+ 2s'} }\rmd x \rmd y \cleq h^{2(s-s')} \lfloor u \rfloor_{s, T \cup T'}^2.
\end{equation}

(ii) If $T  \cap T' = \emptyset$, we have, for $x \in T$ and $y \in T'$, an estimate $h \cleq |x-y|$, which yields:
\begin{equation}
\frac{|u(x) - u(y)|^2}{| x - y |^{n+ 2s'} } \cleq  h^{2(s-s')}\frac{|u(x) - u(y)|^2}{| x - y |^{n+ 2s} }.
\end{equation}
We integrate this inequality for $(x,y) \in T \times T'$.

Finally we sum the inequalities obtained in (i) and (ii), and notice that the right hand side is dominated by $h^{2(s-s')} \lfloor u \rfloor_{s}^2$.  
\end{proof}

The following result will be used repeatedly. It shows in particular that elements of $X_h$ are close in $\rmH^s(S)$ to smooth functions. It enables one to transfer arguments valid for smooth fields, to $X_h$.
\begin{proposition} \label{prop:slobstab} For $0< s < 1/2$ and $\delta >0$,  fixing $\epsilon$ small enough, we may obtain, for all $u \in X_h$:
\begin{equation}
\lfloor u - \phi_{\epsilon h} \ast u \rfloor_s \leq \delta \lfloor u \rfloor_s .
\end{equation}
\end{proposition}

\begin{proof}
Pick $\mu >0$ such that $s + \mu < 1/2$. Use the interpolation inequality to write, for $u \in X_h$:
\begin{align}
\lfloor u - \phi_{\epsilon h} \ast u \rfloor_s  & \cleq  | u - \phi_{\epsilon h} \ast u |^{\frac{\mu}{s + \mu}} \ \lfloor u - \phi_{\epsilon h} \ast u \rfloor^{\frac{s}{s + \mu}}_{s+\mu},\\
& \cleq  | u - \phi_{\epsilon h} \ast u |^{\frac{\mu}{s + \mu}} \ \lfloor u \rfloor^{\frac{s}{s + \mu}}_{s+\mu},
\end{align}
using uniform boundedness of convolution by in $\phi_{\epsilon h}$ in Sobolev spaces.

We apply Proposition \ref{prop:deleps} to the first term and  Proposition \ref{prop:slobinv} to the second term.  We get:
\begin{align}
\lfloor u - \phi_{\epsilon h} \ast u \rfloor_s & \cleq  (h^s \lfloor u \rfloor_s)^{\frac{\mu}{s + \mu}}  (h^{-\mu} \lfloor u \rfloor_s)^{\frac{s}{s + \mu}}, \\
& \cleq \lfloor u \rfloor_s.
\end{align}
In this estimate, the constant may be rendered arbitrarily small by picking $\epsilon$ small enough, due to Proposition \ref{prop:deleps}.
\end{proof}
We have a second inverse inequality for Slobodetskij semi-norms.
\begin{proposition} \label{prop:slobl2}  For $0< s < 1/2$ we have an estimate, for all $u \in X_h$:
\begin{equation}
\lfloor u \rfloor_s \cleq h^{-s}  | u |.
\end{equation}
\end{proposition}
\begin{proof}
We write, using Proposition \ref{prop:slobstab}, an interpolation inequality, an inverse inequality and Proposition \ref{prop:l2stab}:
\begin{align}
\lfloor u \rfloor_s & \cleq \lfloor \phi_{\epsilon h} \ast u \rfloor_{s}, \\
 & \cleq | \phi_{\epsilon h} \ast u |^{1-s} | \nabla (\phi_{\epsilon h} \ast u) |^{s}, \\
 & \cleq | \phi_{\epsilon h} \ast u |^{1-s} \,  h^{-s}  |u|^{s}, \\
 & \cleq h^{-s}  | u |,
\end{align}
as announced.
\end{proof}
Here is a third type of inverse inequality.
\begin{proposition}\label{prop:invds}
 For $0< s < 1/2$ , we have an inverse inequality, for $u\in X_h$:
\begin{align}
| \rmd u |  & \cleq  h^{s-1}  \lfloor u \rfloor_s.
\end{align}
\end{proposition}
\begin{proof} By scaling from a reference element.
\end{proof}


Here is a strengthening of Proposition \ref{prop:invds}, proved in the same way.
\begin{proposition}
For $0<s<1/2$ we have an inverse inequality, for $u\in X_h$:
\begin{align}
\lceil u \rceil_h  & \cleq  h^{s-1}  \lfloor u \rfloor_s.
\end{align}
\end{proposition}

\begin{proposition} \label{prop:discint} 
For $0<  s < 1/2$, we have a discrete interpolation estimate of the form, for $u \in X_h$:
\begin{equation}
\lfloor u  \rfloor_s \cleq |u|^{1-s} \lceil u \rceil^{s}_h.
\end{equation}
\end{proposition}
\begin{proof}
We write, using Proposition \ref{prop:slobstab}, an interpolation inequality, Proposition \ref{prop:l2stab} and Proposition \ref{prop:oneequiv}:
\begin{align}
\lfloor u \rfloor_s & \cleq  \lfloor  \phi_{\epsilon h} \ast u \rfloor_s, \\ 
&\cleq  | \phi_{\epsilon h} \ast u |^{1-s} | \nabla (\phi_{\epsilon h} \ast u) |^{s}, \\
& \cleq | u |^{1-s} \lceil u \rceil_h^{s},
\end{align}
as announced.
\end{proof}

We deduce:
\begin{proposition}\label{prop:spstabs} 
The mollified interpolator (smoothed projector) $\Pi_h$ is stable in $\rmH^{s'}(S) \to \rmH^s(S)$, for $0<s< s'< 1/2$.
\end{proposition}
\begin{proof}
We write, for $u \in \rmH^1(S)$, using Proposition \ref{prop:discint} and Propostion \ref{prop:smstabh1}:
\begin{align}
\lfloor \Pi_h u \rfloor_s & \cleq |\Pi_h u|^{1-s}  \lceil \Pi_h u \rceil^{s}_h, \\
& \cleq | u|^{1-s}  | \nabla u |^{s}. 
\end{align}
One concludes by interpolation theory (see Lemma 25.3 in \cite{Tar07}).
\end{proof}

\paragraph{Approximation orders.}

To obtain convergence estimates in $\rmH^s(S)$ we use:
\begin{lemma}\label{lem:sobdom}
We have:
\begin{equation}\label{eq:sobdom}
\lfloor u \rfloor_{s, S}^2 \cleq  \sum_{T\in \calT_h^n} \lfloor u \rfloor_{s, \calM(T)}^2 + \frac{1}{sh^{2s}} | u |_S^2,
\end{equation}
where we sum over $n$-cells $T$ in $\calT_h$, given that $\calM(T)$ denotes the macro-element surrounding $T$.
\end{lemma}

\begin{proof}
We first prove, with constants independent of $h$:
\begin{equation}\label{eq:sobloc}
\lfloor u \rfloor_{s,S}^2 \cleq  \iint_{\tiny 
\begin{array}{l}
(x,y) \in S \times S\\
|x - y | \leq h
\end{array} } 
\frac{|u(x) - u(y)|^2}{| x - y |^{n + 2s } }\rmd x \rmd y
+\frac{1}{s h^{2s}} \| u\|^2_{\rmL^2(S)}.
\end{equation}
In the case $S = \bbR^n$ we write:
\begin{equation}
\iint_{\tiny
\begin{array}{l}
(x,y) \in S \times S\\
|x - y | \geq h
\end{array} } 
\frac{|u(x) - u(y)|^2}{| x - y |^{n+ 2s } }\rmd x \rmd y = \int_{|z| \geq h} \frac{\| u - \tau_z u \|_{\rmL^2}^2}{|z|^{n+ 2s}}\rmd z.
\end{equation}
Then we remark that:
\begin{equation}
\int_{|z| \geq h} \frac{1}{|z|^{n+ 2s}}\rmd z \cleq \frac{1}{s h^{2s}}. 
\end{equation}
This proves (\ref{eq:sobloc}) when $S = \bbR^n$.

This estimate extends to other domains $S$ because there exist continuous extension operators from Lipschitz subdomains of $\bbR^n$ to $\bbR^n$. 

Then we replace $h$ by $\epsilon h$, for a fixed $\epsilon$ small enough that the $(\epsilon h)$-neighborhood of a cell $T$ is in $\calM(T)$. In the first term on the right hand side of (\ref{eq:sobloc}), we may let the integration domain consist of pairs $(x,y) \in T \times \calM(T)$, for $T \in \calT_h$. This integral is in turn dominated by the first term on the right hand side of (\ref{eq:sobdom}). 
\end{proof}
The point of this lemma is that the norms on the right hand side are local enough that standard scaling techniques from reference macro-elements may be used.

Recall that we suppose that our finite element spaces contain polynomials up to degree $p$. We get:
\begin{proposition}\label{prop:ratepi}
 For $0 < s < s'< 1/2$ we have:
\begin{equation}
\lfloor u - \Pi_hu \rfloor_s \cleq \lfloor u - I_h (\phi_{\epsilon h} \ast u) \rfloor_{s'} \cleq h^{p+1 - s'} | \nabla^{p+1} u|.
\end{equation}
\end{proposition}
\begin{proof}
The first inequality follows from Proposition \ref{prop:spstabs}.  

The second inequality is based on Lemma \ref{lem:sobdom}, as follows. We look at the two terms on the right hand side of (\ref{eq:sobdom}) separately.

When $T$ is an $n$-cell in $\calT_h$ we denote by $\calM^2(T)$ the macro-element of order $2$, defined by adding one more layer of elements around $\calM(T)$. We have, by scaling:
\begin{equation}
\lfloor u - I_h (\phi_{\epsilon h} \ast u) \rfloor_{s', \calM(T)} \cleq h^{p+1 - s'} | \nabla^{p+1} u|_{\calM^2(T)}.
\end{equation}
This is combined with the more standard estimate:
\begin{equation}
\frac{1}{h^{s'}}| u - I_h (\phi_{\epsilon h} \ast u) | \cleq h^{p+1 - s'} | \nabla^{p+1} u|.
\end{equation}
This completes the proof.
\end{proof}

\begin{proposition}\label{prop:ratepid}
We have, for $s \in ]0,1[$, and $u \in \rmH^{p+1}\alter^\bs(S)$:
\begin{equation}
\| \rmd u - \rmd \Pi_h u\|_{\rmH^{-s}} \cleq h^{p+s} | \nabla^{p+1} u |.
\end{equation}
\end{proposition}
\begin{proof}
We write:
\begin{align}
| \rmd u - \rmd \Pi_h u| & = | \rmd u - \Pi_h \rmd u|,\\
& \cleq h^{p} | \nabla^{p} \rmd u |,\\
& \cleq h^{p} | \nabla^{p+1} u |.
\end{align}
We also have:
\begin{align}
\| \rmd u - \rmd \Pi_h u\|_{\rmH^{-1}} & \cleq  | u - \Pi_h u|,\\
& \cleq h^{p+1} | \nabla^{p+1} u |.
\end{align}
One then concludes by interpolation theory.
\end{proof}


\paragraph{Fractional estimates on discrete Hodge decompositions.}

\begin{proposition}\label{prop:discdomv}
For $0< s <s' < 1/2$ we have an estimate, for $u\in V_h$:
\begin{equation}
\lfloor u \rfloor_s \cleq  \| \rmd u \|_{\rmH^{s'-1}} \cleq |\rmd u|.
\end{equation}
\end{proposition}
\begin{proof}

We use Proposition \ref{prop:slobl2}, the stability of $\Pi_h$ in $\rmL^2$ and also Proposition \ref{prop:spstabs}, Lemma \ref{lem:mag} and the error estimate for $\Pi_h$  for $\rmH^s(S)$ data, and finally the continuous theory of the Hodge decomposition. For $u \in V_h$:
\begin{align}
\lfloor u \rfloor_s & = \lfloor \Pi_h u \rfloor_s,\\
& \leq \lfloor \Pi_h (u -\frH u) \rfloor_s + \lfloor \Pi_h \frH u \rfloor_s,\\
& \cleq h^{-s} | \Pi_h (u -\frH u) | + \lfloor \Pi_h \frH u \rfloor_s,\\
& \cleq h^{-s} | u -\frH u | + \lfloor \frH u \rfloor_{s'},\\
& \cleq h^{-s} | \frH u -\Pi_h \frH u) | + \lfloor \frH u \rfloor_{s'},\\
& \cleq  \lfloor \frH u \rfloor_{s'},\\
& \cleq \| \rmd u \|_{\rmH^{s'-1}},\\
& \cleq |\rmd u|.
\end{align}
This concludes the proof.
\end{proof}

The following proposition reflects the proximity of $G_h$ to $G$ and the equivalence of norms on the finite dimensional space $G$.
\begin{proposition} \label{prop:discdomb} For $0<s <1/2$ we have an estimate, for $u\in G_h$:
\begin{equation}
\lfloor u \rfloor_s \cleq |u|.
\end{equation}
\end{proposition}
\begin{proof} This proof is designed so as to extend to general $S$. Remark \ref{rem:ggh} would provide an alternative trivial proof for flat toruses.

Let $Q_h$ denote the $\rmL^2$ projection onto $W_h$. We remark that it maps $G$ to $G_h$. We also use that $\Pi_h$ maps $W$ to $W_h$, which follows from commutation with the exterior derivative. We may write for $u\in G$:
\begin{equation}
|u -Q_h u|  \leq |u -\Pi_h u | \cleq h |\nabla u|.
\end{equation}
Since $G$ and $G_h$ have the same dimension, one deduces that $Q_h : G \to G_h$ is eventually invertible with an inverse which is uniformly bounded $O \to O$.

Then we write, for $u \in G$, using Propositions \ref{prop:slobl2} and \ref{prop:spstabs}, equivalence of norms on $G$ and the above remark:
\begin{align}
\lfloor Q_h u \rfloor_s & \leq \lfloor Q_h (u  - \Pi_h u) \rfloor_s  + \lfloor Q_h \Pi_h u \rfloor_s ,\\
& \cleq h^{-s} | Q_h (u  -  \Pi_h u )| + \lfloor \Pi_h u \rfloor_s,\\
& \cleq h^{-s} | u  - \Pi_h u | + \lfloor u \rfloor_{s'}, \textrm{ with } s < s'< 1/2,\\
& \cleq \lfloor u \rfloor_{s'},\\
&  \cleq | u |,\\
& \cleq | Q_h u |.
\end{align}
This concludes the proof.
\end{proof}

The following proposition shows that solutions to discrete problems may have an improved regularity. For instance, for $0$-forms, it provides $\rmH^{1+s}(S)$ estimates for finite element approximations to the Laplace equation.
\begin{proposition} \label{prop:discdomw}
For $0 < s <1/2$ we have an estimate, for $u\in V_h$:
\begin{equation}
\lfloor \rmd u \rfloor_s \cleq \sup_{v \in X_h} \frac{ | \langle \rmd u , \rmd v \rangle |}{| v|}.
\end{equation}
\end{proposition}
\begin{proof}
We write, for $u\in V_h$, using Proposition \ref{prop:slobstab} and elliptic regularity:
\begin{align}
\lfloor \rmd u \rfloor_s & \ceq \lfloor \rmd (\phi_{\epsilon h} \ast u) \rfloor_s, \\
& \ceq | \rmd^\star \rmd (\phi_{\epsilon h} \ast u)|_{\rmH^{s-1}}.
\end{align}
From this estimate we continue:
\begin{align}
\lfloor \rmd u \rfloor_s & \cleq \sup_{v \in \overline V \cap \rmH^{1-s}(S)} \frac{\langle \rmd  (\phi_{\epsilon h} \ast u) , \rmd v \rangle}{|v| + \lfloor v \rfloor_{1-s}},\\
& \cleq \sup_{v \in \overline V \cap \rmH^{1-s}(S)} \frac{\langle \rmd  u , \rmd   (\phi_{\epsilon h} \ast v) \rangle}{|v| + \lfloor v \rfloor_{1-s}},\\
& \cleq \sup_{v \in \overline V \cap \rmH^{1-s}(S)} \frac{\langle \rmd  u , \rmd   P_h (\phi_{\epsilon h} \ast v) \rangle}{|v| + \lfloor v \rfloor_{1-s}}.
\end{align}
Next we write, for $v \in \overline V \cap \rmH^{1-s}(S)$, using Proposition \ref{prop:phan}:
\begin{align}
| \phi_{\epsilon h} \ast v  -  P_h (\phi_{\epsilon h} \ast v) | & \cleq h \, | \rmd (\phi_{\epsilon h} \ast v) |,\\
& \cleq h^{1-s} | \rmd v |_{\rmH^{-s}(S)},\\
& \cleq h^{1-s} \lfloor v \rfloor_{1-s}.
\end{align}
We deduce:
\begin{align}
|  P_h (\phi_{\epsilon h} \ast v) | & \cleq  | \phi_{\epsilon h} \ast v|   +  h^{1-s} \lfloor v \rfloor_{1-s},\\
& \cleq | v|  + \lfloor v \rfloor_{1-s}.
\end{align}
So we obtain:
\begin{equation}
\lfloor \rmd u \rfloor_s \cleq \sup_{v \in V_h} \frac{ | \langle \rmd u,  \rmd v \rangle |}{| v|}.
\end{equation}
This is a little stronger than the announced result.
\end{proof}

\begin{remark}\label{rem:dischod}
The results of this section give in particular estimates for elements of $V_h$ that mimick similar estimates known to hold for elements of $V$. Various results along these lines have been obtained the last decade.
\begin{itemize}
\item  The property that for $u \in V_h$, $| u| \cleq |\rmd u|$ is known as the Poincar\'e-Friedrichs estimate, and was proved for closed manifolds in \cite{Chr07NM}. It is equivalent to the wellposedness of equations related to the Hodge-Laplace \cite{ArnFalWin06}.

\item The stronger result, that for $u \in V_h$, $| u| \cleq |\rmd u|_{\rmH^{-1}}$ was proved in \cite{Chr09Calc} in 3D, and used there to prove a discrete div-curl lemma. It can in general be deduced from the existence of $\rmL^2$-bounded commuting projections.
 
\item The property that for $u_h \in V_h$, if $\rmd u_h$ is bounded in $\rmL^2$, then a subsequence of $(u_h)$ converges in $\rmL^2$, is known as discrete compactness, in the sense of Kikuchi \cite{Kik89}. It is used for proving convergence of eigenvalue problems in particular for Maxwell's equations \cite{Bof10}. This property is discussed in general, in relation to the previous two variants of Poincar\'e-Friedrichs estimates, in \cite{ChrWin13IMA}. Discrete compatness is equivalent to the gap property $\delta(V_h, V) \to 0$ in the norm of $X$. The  gap property  has applications outside spectral theory, for instance to prove inf-sup conditions for non-coercive problems, see \cite{BufChr03}\cite{Buf05} and Lemma 4.4.1 in \cite{Chr02PhD}.

\item For $u \in V_h$, the $\rmL^p$-norm of $u$ is bounded by the $\rmL^2$ norm of $\rmd u$, for the exponents $p$ such that we have a Sobolev embedding $\rmH^1(S) \to \rmL^p(S)$. For differential forms this result is included in \cite{ChrMunOwr11} (Propositions 5.69 and 5.70). Such discrete Sobolev embeddings were used to study non-linear problems in 2D in \cite{ChrSch11}.

\end{itemize}
\end{remark}

\begin{remark}
In \cite{DroEymGalHer13} a quite general framework for analysing numerical schemes is proposed, mainly in the context of scalar equations but including mixed methods. It encompasses both non-conforming Galerkin methods and finite volume like schemes, and involves in particular an axiomatic approach to discrete compactness properties. It is quite different from discrete compactness in the sense of Kikuchi. On the other hand, that discrete compactness in their sense (see Definition 2.8 of \cite{DroEymGalHer13}) holds for the present type of discretizations is proved in Proposition 5.71 of \cite{ChrMunOwr11}, at least for the $\rmL^2$ setting.
\end{remark}

\section{Abstract discretization of the Hodge-Dirac \label{sec:hoddir}}

For the Hodge-Dirac discretized with differential forms, $X_h$ is not a subspace of $Y$ and $X$ is not compact in $O$. Contrary to the case of Maxwell problems, the operator is unsigned, even after shifting by multiples of the identity. Below we provide ways around these difficulties. A first theory is given, showing how far one can get with a theory based on $X_h$ interpreted as a subspace of $X$. Then we provide second theory, where $X_h$  is viewed as a non-conforming discretization of $Y$. This theory allows for lower order pertubations, as detailed in the third part of this section.

\paragraph{A first general theory.}

We suppose we have a separable Hilbert space $O$, with scalar product $\langle \cdot, \cdot \rangle$ and norm $| \cdot |$. We suppose we have another Hilbert space $X$ which is dense and compactly embedded in $O$. We are given a continuous operator $\rmd: X \to O$, such that the following bilinear form on $X$: 
\begin{equation}
(u, v) \mapsto \langle \rmd u, \rmd v \rangle + \langle u, v \rangle,
\end{equation}
defines the norm of $X$, or at least an equivalent one. We suppose that $\rmd$ maps $X$ to $X$ and in fact that $\rmd^2 = 0$.

We let $W$ denote the kernel of $\rmd$ on $X$. We let $V$ be the orthogonal of $W$ in $X$ with respect to the scalar product of $O$, or, equivalently, that of $X$. We suppose that the injection of $V$ in $O$ is compact.

We let $a$ denote the continuous symmetric bilinear form on $X$ defined by:
\begin{equation}
a(u,v) = \langle \rmd u, v \rangle + \langle u, \rmd v \rangle.
\end{equation}

We are interested in the eigenmode problems, find $u\in X$ and $\lambda \in \bbR$ such that, for all $v\in X$:
\begin{equation}
a(u,v) = \lambda \langle u, v \rangle.
\end{equation}
The kernel of $a$, or equivalently the $0$-eigenspace, is:
\begin{equation}
G = \{u \in X \ : \ \rmd u = 0 \textrm{ and } u \perp \rmd X\}.
\end{equation}
We suppose that it is finite-dimensional.

To study these eigenmode problems, we introduce the bilinear form $b$ on $X$ defined by:
\begin{equation}
b(u,v) = \langle \rmd u, \rmd v \rangle.
\end{equation}
This bilinear form is semi-positive and its kernel is $W$. The orthogonal of the kernel is $V$, which is compactly embedded in $O$. We are therefore in the setting of \cite{ChrWin13IMA}. The non-zero eigenvalues are positive, constitute an increasing and diverging sequence and have finite dimensional associated eigenspaces.

We let $\mu[k]$, $k \in \bbN$, denote the increasing sequence of strictly positive eigenvalues of $b$, counted with multiplicities and with a corresponding choice of eigenvector $v[k]$. We suppose that they are orthonormalized in $O$, so they constitute a Hilbertian basis of $\overline V$, the closure of $V$ in $O$. We let $\lambda[k] = \mu[k]^{1/2}$. We remark that the vectors $\lambda[k]^{-1} \rmd  v[k]$ constitute a Hilbertian basis of the subspace of $W$ orthogonal to $G$.

Using that $v[k] \perp W$  and $\rmd X \subset W$ we compute, for $w \in X$:
\begin{align}
a(v[k] \pm \lambda[k]^{-1} \rmd  v[k], w) & = \langle v[k] \pm  \lambda[k]^{-1} \rmd  v[k], \rmd w \rangle + \langle \rmd v[k], w\rangle,\\
& = \langle \pm  \lambda[k]  v[k], w \rangle + \langle \rmd v[k], w\rangle.
\end{align}
We get:
\begin{itemize}
\item $+\lambda[k]$ is an eigenvalue of $a$ for the eigenvector $ v[k] + \lambda[k]^{-1} \rmd  v[k]$,
\item $- \lambda[k]$ is an eigenvalue of $a$ for the eigenvector $ v[k] - \lambda[k]^{-1} \rmd  v[k]$.
\end{itemize}

Together, these two families of vectors constitute a Hilbertian basis of the subspace of $O$ orthogonal to $G$, after a scaling by $\sqrt 2$.

We suppose that we are given a sequence of finite-dimensional subspaces $X_n$ of $X$, $n \in \bbN$. We then look at the discrete eigenvalue problems: find $u_n\in X_n$ and $\lambda_n \in \bbR$ such that, for all $w_n\in X_n$:
\begin{equation}
a(u_n,w_n) = \lambda_n \langle u_n, w_n \rangle.
\end{equation}
We may also study the discrete eigenvalue problems for $b$.

We suppose that $\rmd$ maps $X_n$ to $X_n$. Then the algebraic properties, of the relation between the eigenvalue problems for $a$ and $b$, carry through exactly as above, to the discrete setting.

For the analysis, we suppose that each $X_n$ is equipped with a projector  $\Pi_n : O \to X_n$, that these are uniformly bounded in $O$, commute with $\rmd$ and converge pointwise to the identity. Under these hypotheses the discrete eigenvalue problem for $b$ falls under the scope of \cite{ChrWin13IMA}. There, convergence is proved for the discrete eigenmodes of $b$, and we deduce that convergence also holds for the discrete eigenmodes of $a$. 

The problem is that, as far as we can see, this analysis does not allow one to perturb $a$ by a bilinear form $c$ continuous on $O$, because $X$ is not compactly embedded in $O$. The bilinear form $c$ would model $0$-order terms in the Dirac equation, such as the ones produced by the electromagnetic field. It would also be of interest to weaken the hypothesis $\rmd^2 = 0$ to the following: $\rmd^2$ extends to a bounded operator $O \to O$. This would cover the case where $\rmd$ is a covariant exterior derivative and $\rmd^2$ represents curvature terms.

\paragraph{A second general theory.}

We suppose we have a separable Hilbert space $O$, with scalar product $\langle \cdot, \cdot \rangle$ and norm $| \cdot |$. We suppose we have another Hilbert space $Y$ which is dense and compactly embedded in $O$. We are given a bilinear form $a$ which is continuous  and symmetric on $Y\times Y$ and extends continuously  to $Y\times O$ and $O \times Y$.  We define an operator $A: Y \to O$ by, for all $u \in Y$ and all $v \in O$:
\begin{equation}
\langle A u , v \rangle = a( u,v).
\end{equation}
We make the hypothesis that $A: Y \to O$ is Fredholm. By symmetry, its index is necessarily $0$. Care is taken so that the theory covers the case where $A$ is unsigned.

From the point of view of computations, the difficulty is to construct ''nice'' subspaces of $Y$. For instance, in the first theory developed above, ''nice'' meant equipped with bounded projectors commuting with the differential $\rmd$.  These were not subspaces of $\rmH^1\alter^\bs(S)$, not even of $\rmH^{1/2}\alter^\bs(S)$ . Now we consider a theory without explicit mention of such a differential. We will consider discretization spaces $X_n$ which are not subspaces of $Y$. Since then $a$ cannot be simply restricted to $X_n$, we adopt the setting of non-conforming methods. We suppose we have finite dimensional subspaces $X_n$ of $O$, equipped with symmetric bilinear forms $a_n$, and proceed to define suitable consistency and stability requirements. 

We suppose that we have a space $Z$, intermediate between $O$ and $Y$, with continuous inclusions and such that the injection $Z \to O$ is compact, and such that for all $n$, $X_n \subseteq Z$. 

\begin{remark} \label{rem:choicez} 
For instance, a choice of $Z$ could be of the form: 
\begin{equation}
Z \subseteq [O, Y]_s, \textrm{ for some } s \in ]0, 1]. \label{eq:zees}
\end{equation}
For the Hodge-Dirac, one can use a fractional order Sobolev space:
\begin{equation}
Z^k =  \rmH^s\alter^k(S) \textrm{ for some } s \in ]0, 1/2[.
\end{equation}
Notice that the boundary condition does not make sense in this norm. On the other hand this norm is weak enough to allow finite element spaces without full inter-element continuity. 

We will also use the space:
\begin{equation}
Z^k =  \{u \in \rmH^{s_1}\alter^k(S) \ : \ \rmd u \in  \rmH^{-s_2}\alter^k(S) \} \textrm{ for some } s_1, s_2 \in ]0, 1/2[.
\end{equation}
\end{remark}

The forms $a_n$ are required to be consistent with $a$ in the following sense. 
\begin{definition}\label{def:con} We say that the forms $a_n$ are \emph{consistent} with $a$ when the following holds. 
There is a dense subset $Y^\infty$ of $Y$ such that for all $u \in Y^\infty$ there is a sequence $u_n \in X_n$ such that $u_n \to u$ in $O$ and:
\begin{equation}\label{eq:conan}
\lim_{n \to \infty} \sup_{v \in X_n}\frac{|a(u, v) - a_n(u_n, v)|}{\| v \|_Z} = 0.
\end{equation}
\end{definition}

We also need a stability estimate for $a_n$. The following one is reasonable when $A$ is injective.
\begin{definition}\label{def:stab}
We say that we have a \emph{weak stability}, when the following holds, uniformly in $n$:
\begin{equation}\label{eq:isan}
1 \cleq \infsup{u \in X_n}{v\in X_n}{|a_n(u,v)|}{|u|\, \|v\|_Z}. 
\end{equation}
\end{definition}

Suppose that $A$ is injective, so that it is invertible $Y \to O$. Let $K$ be the inverse. Equivalently, the operator $K: O \to Y$ is defined by, for all $u,v \in O$:
\begin{equation}\label{eq:akdef}
a(Ku, v) = \langle u, v \rangle.
\end{equation}
As an operator $O\to O$, $K$ is symmetric and compact, but unsigned. The spectral theorem for compact self-adjoint operators applies. Under the weak inf-sup condition, one can also define an operator $K_n: O \to X_n$, by, for all $v \in X_n$:
\begin{equation}\label{eq:akhdef}
a(K_n u, v) = \langle u, v\rangle.
\end{equation}
Convergence of the discrete eigenmode problem is obtained by comparing $K_n$ with $K$. The following estimate appears to be the most convenient sufficient condition for convergence.
\begin{proposition}\label{prop:compconv} Suppose $A$ is injective.
Under the conditions of consistency and stability of the preceding definitions, we have convergence in operator norm:
\begin{equation}\label{eq:compconv}
\| K - K_n \|_{O \to O} \to 0.
\end{equation}
\end{proposition}

\begin{proof}
Define the space $Z'$ to be the dense subspace of $O$ which is dual of $Z$ with respect to $a$. In other words, choose $Z'$ so that bilinear form $a$ extends to a continuous and invertible one on $Z' \times Z$.
For instance, in the case of equality in (\ref{eq:zees}), we would have: 
\begin{equation}
Z' = [O, Y]_{1-s}.
\end{equation}

Define an operator $P_n: Z' \to X_n$, by, for all $v\in X_n$:
\begin{equation}
a_n(P_n u, v) = a( u, v ).
\end{equation}
From the discrete inf-sup condition, the sequence $P_n$ is uniformly bounded $Z' \to O$.

Moreover for $u \in Y^\infty$, choose $u_n \in X_n$ such that $u_n \to u$ in $O$ and (\ref{eq:conan}) holds. We have:
\begin{align}
|P_n u - u_n| \cleq \sup_{v\in X_n} \frac{|a_n(P_n u -u_n, v)|}{\|v \|_Z} = \sup_{v\in X_n} \frac{|a(u,v) - a_n(u_n, v)|}{\|v \|_Z} \to 0.
\end{align}
Hence $P_n u \to u$ in $O$. Combining with the above uniform boundedness, it follows that $P_n \to I$ pointwise, as operators  $Z' \to O$.

Since the injection $Z \to O$ is compact, $K: O \to Z'$ is compact. 

Therefore $K_n = P_n K$ converges to $ K$ in operator norm $O \to O$.
\end{proof}

\paragraph{Perturbations in the second general theory.} 

We consider the second abstract setting, but notations are compatible with the first one also.

In practice the operators $A : Y \to O$ we consider will not be invertible. Even if they were, we would be interested in perturbations of $A$ of the form $A+C$ where $C: O\to O$ is bounded and symmetric. For these, invertibility would not be guaranteed. Notice however that by adding multiples of the identity $I: O \to O$, invertibility may be achieved, and that such perturbations do not alter the eigenmode problem other than by a shift in the eigenvalues.

Let $G$ denote the kernel of $ A$, which, we remind, is finite-dimensional. We suppose that the kernel $G_n$ of $a_n$ satisfies $\hat \delta (G_n, G) \to 0$, in $O$-norm. In particular they must eventually have the same dimension. Notice that for differential forms and the Hodge-Dirac, this hypothesis is guaranteed by Proposition \ref{prop:basic}.

Let $C:O \to O$ be bounded and symmetric. We let $c$ be the associated bilinear form on $O$, defined by:
\begin{equation}
c(u,v) = \langle C u, v \rangle.
\end{equation}
Moreover we suppose that we have symmetric bilinear forms $c_n$ on $X_n$, with a consistency estimate of the form, for all $u,v \in X_h$:
\begin{equation}\label{eq:cnc}
|c(u,v) - c_n(u,v)| \leq \epsilon_n | u | \, \| v \|_Z,
\end{equation}
for some sequence of positive reals $(\epsilon_n)$ converging to $0$.

\begin{proposition}\label{prop:pertac}   We recall the hypothesis $\hat \delta (G_n, G) \to 0$. We denote $\tilde X_n = X_n \cap G_n^\perp$ and suppose that $a_n$ satisfies:
\begin{equation}\label{eq:istt}
\inf_{u\in \tilde X_n}\sup_{v \in X_n} \frac{| a_n (u,v) |}{| u| \|v\|_Z} \cgeq 1.
\end{equation}
Suppose that $A+ C : Y \to O$ is injective. Then we have the discrete weak stability, for $n$ large enough:
\begin{equation}
\inf_{u\in  X_n}\sup_{v \in  X_n} \frac{| a_n (u,v) + c_n(u,v) |}{| u | \|v\|_Z} \cgeq 1.
\end{equation}
\end{proposition}
\begin{proof}
We proceed by contradiction, and suppose the condition is not satisfied. Choose a subsequence $u_n \in X_n$ such that $ | u_n |  = 1$, and:
\begin{equation}
\sup_{v \in  X_n} \frac{| a_n (u_n,v) + c_n(u_n,v) |}{ \|v\|_Z} \to 0.
\end{equation}
Upon extracting subsequences, we may suppose in addition that $(u_n)$ converges weakly in $O$ to some $u\in O$. 
It follows from the two consistency estimates that for $v \in Y^\infty$:
\begin{equation}
a(u, v) + c(u,v) = 0.
\end{equation}
Therefore $u=0$. 

By consistency (\ref{eq:cnc}) we deduce:
\begin{equation}
\sup_{v \in X_n} \frac{|c_n(u_n, v)|}{\| v\|_Z} \to 0,
\end{equation} 
hence:
\begin{equation}
\sup_{v \in  X_n} \frac{| a_n (u_n,v) |}{ \|v\|_Z} \to 0.
\end{equation}

We may decompose:
\begin{equation}
u_n = f_n + g_n,
\end{equation}
with $f_n \in \tilde X_n$ and $g_n \in G_n$.

Since $\hat \delta (G_n, G) \to 0$, we get that both $f_n$ and $g_n$ converge weakly to $0$ in $O$. Actually $g_n$ converges strongly in $O$, since $G$ is finite-dimensional.

So $|f_n| \to 1 $ and:
\begin{equation}
\sup_{v \in  X_n} \frac{| a_n (f_n,v) |}{ \|v\|_Z} \to 0.
\end{equation}
This contradicts (\ref{eq:istt}).
\end{proof}

In the above considerations we could replace boundedness of $C:O \to O$ by boundedness $C:Z \to O$, which is somewhat weaker. 

\begin{remark}
It can be noticed that the principles behind Proposition \ref{prop:pertac} are that, under suitable conditions, firstly, inf-sup conditions on the orthogonal of some well-behaved finite dimensional subspaces provide full inf-sup conditions for injective operators, and secondly, the property of having an inf-sup condition on the orthogonal of the kernel is stable under compact perturbations. Variants of these principles were developed under slightly different hypotheses in \cite{Chr04MC}, in particular we did not consider weak norms there. On the other hand, the use of weak norms is in \cite{Chr11NM}, but there special care was given to the fact that the kernels were infinite-dimensional.
\end{remark}

\paragraph{Strengthened stability.}

The previous two paragraphs are natural from the point of view of establishing the convergence $K_h \to K$ in the operator norm $O \to O$, as in Proposition \ref{prop:compconv}. However, to study the forward problem in terms of stability and convergence rates, the techniques yield suboptimal results. In this paragraph we indicate how an almost optimal stability estimate can be obtained. It is the basis for establishing almost optimal convergence rates for the forward problem, which in turn provide convergence rates for the eigenvalue problem \cite{BabOsb89}\cite{BanOsb90}\cite{BabOsb94}. 

The functional setting is as in the second general theory.  We suppose that we have a space $Z$, intermediate between $O$ and $Y$, with continuous inclusions and such that the injection $Z \to O$ is compact. We also suppose that we have finite dimensional subspaces $X_n$ of $Z$, equipped with symmetric bilinear forms $a_n$ and $c_n$.

We make the following hypotheses:
\begin{itemize}
\item There is a dense subset $Y^\infty$ of $Y$ such that for all $u \in Y^\infty$ there is a sequence $u_n \in X_n$ such that $u_n \to u$ in $Z$ (not only $O$ as in Definition \ref{def:con}) and:
\begin{equation}\label{eq:conan}
\lim_{n \to \infty} \sup_{v \in X_n}\frac{|a(u, v) - a_n(u_n, v)|}{\| v \|_Z} = 0.
\end{equation}

\item As before the kernel of $a$ is denoted $G$, the kernel of $a_n$ is denoted $G_n$. We suppose that 
$\delta(G_n,G) \to 0$ in the norm of $Z$ (before the norm of $O$ was used).

\item The bilinear form $c_n$ is consistent with $c$ and in the sense that for a sequence $(\epsilon_n)$ of positive reals converging to $0$ we have, for all $u,v \in X_h$:
\begin{equation}\label{eq:cncw}
|c(u,v) - c_n(u,v)| \leq \epsilon_n \| u \|_Z \, \| v \|_Z.
\end{equation}
This condition is weaker than (\ref{eq:cnc}).
\end{itemize}

Under these circumstances we obtain the following variant of Proposition \ref{prop:pertac}. The proof is almost identical.
\begin{proposition}\label{prop:pertacm} 
We denote $\tilde X_n = X_n \cap G_n^\perp$ and suppose that $a_n$ satisfies:
\begin{equation}
\inf_{u\in \tilde X_n}\sup_{v \in X_n} \frac{| a_n (u,v) |}{\| u\|_Z \|v\|_Z} \cgeq 1.
\end{equation}
Suppose that $A + C : Y \to O$ is injective. Then we have, for $n$ large enough:
\begin{equation}
\inf_{u\in  X_n}\sup_{v \in  X_n} \frac{| a_n (u,v) + c_n(u,v) |}{\| u \|_Z \|v\|_Z} \cgeq 1.
\end{equation}
\end{proposition}
\begin{proof} 
We proceed by contradiction, and suppose the condition is not satisfied. Choose a subsequence $u_n \in X_n$ such that $ \| u_n \|_Z  = 1$, and:
\begin{equation}
\sup_{v \in  X_n} \frac{| a_n (u_n,v) + c_n(u_n,v) |}{ \|v\|_Z} \to 0.
\end{equation}
Upon extracting subsequences, we may suppose in addition that $(u_n)$ converges weakly in $Z$ to some $u\in Z$. 
It follows from the consistency estimates that for $v \in Y^\infty$:
\begin{equation}
a(u, v) + c(u,v) = 0.
\end{equation}
Therefore $u=0$. 

By consistency (\ref{eq:cncw}) we deduce:
\begin{equation}
\sup_{v \in X_n} \frac{|c_n(u_n, v)|}{\| v\|_Z} \to 0,
\end{equation} 
hence:
\begin{equation}
\sup_{v \in  X_n} \frac{| a_n (u_n,v) |}{ \|v\|_Z} \to 0.
\end{equation}

We may decompose:
\begin{equation}
u_n = f_n + g_n,
\end{equation}
with $f_n \in \tilde X_n$ and $g_n \in G_n$.

Since $\hat \delta (G_n, G) \to 0$ in $Z$, we get that both $f_n$ and $g_n$ converge weakly to $0$ in $Z$. Actually $g_n$ converges strongly in $Z$, since $G$ is finite-dimensional.

So $\| f_n\|_Z \to 1 $ and:
\begin{equation}
\sup_{v \in  X_n} \frac{| a_n (f_n,v) |}{ \|v\|_Z} \to 0.
\end{equation}
This contradicts (\ref{eq:istt}).
\end{proof}

\section{Concrete discretization of the Hodge-Dirac \label{sec:conc}}

In this section we discretize the Hodge-Dirac operator of Section \ref{sec:func}, with differential forms as detailed in Section \ref{sec:rev}. We adopt notations both from the concrete setting of Section \ref{sec:rev} and the abstract setting of Section \ref{sec:hoddir}.  In particular, the abstract setting is employed with the choices:
\begin{align}
O & = \rmL^2 \alter^\bs(S),\\
X & = \{ u \in O \ : \ \rmd u \in O \},\\
Y & = \rmH^1 \alter^\bs(S),
\end{align}
For the space $Z$ several choices will be relevant. See Remark \ref{rem:choicez} for a list of choices.

Recall that:
\begin{equation}
a(u,v) = \langle \rmd u , v \rangle + \langle u, \rmd v\rangle.
\end{equation}
Since $X_h$ is a subspace of $X$ and $a$ is continuous on $X$ we can let $a_h$ be the restriction of $a$ to $X_h$.

We first obtain convergence for the eigenmode problem, allowing for zero order perturbations of $a$. Then we establish convergence rates.

\paragraph{Convergence.} Recall the definition of the seminorm (\ref{eq:ysem}). The following proposition is almost tautological.
\begin{proposition} \label{prop:discis} Define $\tilde X_h = X_h \cap G_h^\perp$. We have an estimate:
\begin{equation}
1 \cleq \infsup{u \in  \tilde X_h}{v\in \tilde X_h}{|a_h(u,v)|}{|u|\, [v]_h}. 
\end{equation}
\end{proposition}
\begin{proof} 
We define an operator $\rmd^\star_h : O \to V_h$ by requiring, for $u\in O$ and $v \in V_h$:
\begin{equation}
\langle \rmd^\star_h u, v\rangle = \langle u, \rmd v\rangle .
\end{equation}
Notice that with this definition, the equality then holds for $v \in X_h$. We also notice that:
\begin{equation}
 \sup_{v \in X_h} \frac{ |\langle u,  \rmd v\rangle | }{|v|} = |\rmd_h^\star u|.
 \end{equation}
 
Choose $u \in X_h$. Let $v = \rmd u + \rmd^\star_h u$. 
We have:
\begin{align}
\frac{a(u,v)}{|v|} & = \frac{|\rmd u|^2 + | \rmd^\star_h u|^2}{( |\rmd u|^2 + | \rmd^\star_h u|^2 )^{1/2}} = [u]_h.
\end{align}
One concludes by restricting attention to $\tilde X_h$ and using the fact that a map and its transpose have the same norm.
\end{proof}
One could say that the whole point of Section \ref{sec:frac} was to prove that the discrete seminorm (\ref{eq:ysem}) dominates a sufficiently strong true seminorm, namely the Slobodetskij one. Here is a precise statement.

\begin{proposition}\label{prop:discdomx}
For $0< s <1/2$ we have an estimate, for $u\in X_h \cap G_h^\perp$:
\begin{equation}\label{eq:firest}
\lfloor u \rfloor_s \cleq [u]_h.
\end{equation}
and for $u \in X_h$:
\begin{equation}\label{eq:secest}
\lfloor u \rfloor_s \cleq  |u| + [u]_h.
\end{equation}
\end{proposition}
\begin{proof}
Choose $u \in X_h$. We write the discrete Hodge decomposition:
\begin{equation}
 u = \rmd w  + v + g, \textrm{ with } w \in V_h, v \in V_h \textrm{ and } g \in G_h. 
\end{equation}
Then we combine Propositions \ref{prop:discdomv}, \ref{prop:discdomb} and \ref{prop:discdomw}, as follows. We estimate:
\begin{align}
\lfloor u \rfloor_s^2 & \leq 3 \lfloor \rmd w \rfloor_s^2 + 3\lfloor v \rfloor_s^2 + 3\lfloor g \rfloor_s^2,\\
& \cleq (\sup_{u' \in X_h} \frac{| \langle \rmd w, \rmd u'\rangle |}{|u'|})^2 + |\rmd v|^2 + |u|^2,\\
& \cleq  (\sup_{u' \in X_h} \frac{| \langle u, \rmd u'\rangle | }{|u'|})^2 + |\rmd u|^2 + |u|^2, \\
&  \cleq [u]_h^2 + |u|^2.
\end{align}
This gives the second estimate (\ref{eq:secest}). To obtain the first estimate (\ref{eq:firest}) one supposes $g=0$ in the above computations.
\end{proof}

We deduce a stability estimate in the following form.
\begin{proposition} \label{prop:discisoz}
For $0 < s < 1/2$, the following weak inf-sup condition holds:
\begin{equation}
1 \cleq \infsup{u \in \tilde X_h}{v\in \tilde X_h}{|a_h(u,v)|}{|u|\, ( |v| + \lfloor v \rfloor_s)}. 
\end{equation}
\end{proposition}
\begin{proof}
From Propositions \ref{prop:discis} and \ref{prop:discdomx}.
\end{proof}

We also check the consistency requirement of the abstract theory.

\begin{proposition} Consistence in the sense of Definition \ref{def:con} holds. For $u\in \rmH^1\alter^\bs(S)$ such that $\rmd u \in \rmH^1\alter^\bs(S)$ we may find $u_h \in X_h$ such that:
\begin{equation}
\sup_{v \in X_h} \frac{|a(u,v) - a_h(u_h, v)|}{|v| + \lfloor v \rfloor_s} \to 0.
\end{equation}
\end{proposition}

\begin{proof}
For such $u$ define $u_h = \Pi_h u \in X_h$. Since the local spaces attached to cells contain the constants, we have convergence rates:
\begin{align}
|u - u_h|  & \cleq h, \\
 |\rmd u - \rmd u_h|&  \cleq h.
\end{align}
For any $v \in X_h$ we have:
\begin{equation}
a(u,v) - a_h(u_h, v) = \langle \rmd u -\rmd u_h,  v \rangle + \langle u - u_h,  \rmd v \rangle ,
\end{equation}
hence:
\begin{equation}
|a(u,v) - a_h(u_h, v)| \cleq  h ( | v |  + | \rmd v | ).
\end{equation}
Using Proposition \ref{prop:invds} we deduce:
\begin{equation}
|a(u,v) - a_h(u_n, v) | \cleq h |v| + h^s \lfloor v \rfloor_s.
\end{equation}
This concludes the proof.
\end{proof}

Finally, combing all the above results, we apply the general perturbation theory with the choice:
\begin{equation}
Z = \rmH^{s}\alter^\bs(S),
\end{equation}
with $s \in ]0, 1/2[$. We arrive obtain:

\begin{corollary}
Denote by $A:Y \to O$ the Hodge-Dirac operator. Consider a variational discretization with differential forms, as discussed above. Let $C:O \to O$ be a bounded self-adjoint operator. It could represent for instance a smooth enough electromagnetic field in the Dirac equation. Consider a consistent variational discretization of $C$ by symmetric bilinear forms $c_h$ on $X_h$. Choose $\mu \in \bbR$ such that $A + C + \mu I$ is injective. Then the associated discrete eigenmodes converge, in the sense given by Proposition \ref{prop:compconv}.
\end{corollary}

\paragraph{Convergence rates.} The previous results did not provide explicit convergence rates, and should one try to deduce rates from the proofs, one would get suboptimal ones. In this paragraph we establish stronger stability estimates and almost optimal convergence rates.

On the space $\tilde X_h = X_h \cap G_h^\perp$ we denote in this paragraph:
\begin{equation}
[ u ]^2_{h,1}  =   | \rmd u |^2 + ( \sup_{v \in X_h} \frac{ |\langle u,  \rmd v\rangle | }{|v|})^2.
\end{equation}
and:
\begin{equation}
[ u ]^2_{h,0} = |u|^2.
\end{equation}
More generally we introduce, on $\tilde X_h$, the norms $[ u ]^2_{h,s}$ for $0  \leq  s \leq 1$ defined by interpolation between $[ \cdot ]^2_{h,1}$ and  $[ \cdot ]^2_{h,0}$, say by the complex method. The space $\tilde X_h$ equipped with the norm $[ u ]^2_{h,s}$ will be denoted $\tilde X_{h, s}$. Actually we need only the cases $s =0, 1/2, 1$.

\begin{proposition}\label{prop:isdhalf}
We have an estimate:
\begin{equation}
1 \cleq \infsup{u \in  \tilde X_h}{v\in \tilde X_h}{|a(u,v)|}{[u]_{h,1/2}\, [v]_{h,1/2}}. 
\end{equation}
\end{proposition}
\begin{proof}
Introduce the map:
\begin{equation}
\calA_h \mapping{\tilde X_h}{\tilde X_h^\star}{u}{a(u, \cdot)}
\end{equation}
From Proposition \ref{prop:discis} and the symmetry of $a_h$ we get that $\calA_h^{-1}$ is uniformly bounded in operator norm $\tilde X_{h, 1}^\star \to \tilde X_{h, 0}$ and $\tilde X_{h, 0}^\star \to \tilde X_{h, 1}$. By interpolation theory it is uniformly bounded in operator norm $\tilde X_{h, 1/2}^\star \to \tilde X_{h, 1/2}$. This can be expressed with the claimed inf--sup condition.
\end{proof}

\begin{proposition}\label{prop:zdhalf}
For $0 < s < 1/2$ we have an estimate, for $u \in \tilde X_h$:
\begin{equation}
\lfloor u \rfloor_s \cleq [u]_{h,1/2}.
\end{equation}
\end{proposition}
\begin{proof}
Combining Propositions \ref{prop:discint} and \ref{prop:honedom} we get, for $u\in \tilde X_h$:
\begin{equation}
\lfloor u \rfloor_{s} \cleq [u]_{h,0}^{1-s} [u]_{h,1}^{s \phantom{-1}}.
\end{equation}
From there we may deduce, for $s < s' \leq 1 $:
\begin{equation}
\lfloor u \rfloor_{s} \cleq [u]_{h,s'}.
\end{equation}
We choose $s' = 1/2$ and obtain the announced estimate.
\end{proof}

\begin{proposition}\label{prop:iss}
We have an estimate:
\begin{equation}
1 \cleq \infsup{u \in  \tilde X_h}{v\in \tilde X_h}{|a_h(u,v)|}{\lfloor u \rfloor_s \, \lfloor v \rfloor_s}. 
\end{equation}
\end{proposition}
\begin{proof}
From Propositions \ref{prop:isdhalf} and \ref{prop:zdhalf}.
\end{proof}

We will perturb this estimate, by interpolating with the natural norm on $X$, which will be written:
\begin{equation}
\|u\|^2_\rmd = | u|^2 + |\rmd u|^2.
\end{equation}
We notice:
\begin{proposition}\label{prop:isd}
We have an estimate:
\begin{equation}
1 \cleq \infsup{u \in  \tilde X_h}{v\in \tilde X_h}{|a(u,v)|}{\|u\|_\rmd \, \|v\|_\rmd}. 
\end{equation}
\end{proposition}
\begin{proof}
For $u \in \tilde X_h$ write $u = v + \rmd w$ with $v,w \in V_h$. Define $u' = w + \rmd v$. We have:
\begin{align}
\| u'\|_\rmd^2 & = |w|^2 + |\rmd v|^2 + |\rmd w|^2,\\
& \ceq |\rmd v|^2 + |\rmd w|^2,\\
& \ceq |\rmd v|^2 + |\rmd w|^2 + |v|^2 = \|u\|^2_\rmd.
\end{align}
We also have:
\begin{align}
a(u,u') & = \langle \rmd u, \rmd v \rangle + \langle u, \rmd w \rangle, \\
& = |\rmd v|^2 + | \rmd w|^2.
\end{align}
Therefore $u'$ provides a quasi-optimal test function for $u$.
\end{proof}

We require the following interpolation result (where it is understood that we are working on differential forms).
\begin{lemma} We have the interpolation result, for $\theta \in ]0, 1[$:
\begin{equation}
[X, \rmH^s(S)]_\theta = \{ u \in \rmH^{\theta s}(S) \ : \ \rmd u \in \rmH^{\theta(s-1)}(S) \}.
\end{equation}
\end{lemma}
\begin{proof}
Since we are working on a periodic domain, this can be proved by Fourier techniques. This approach is detailed in \cite{HipLiZou12} for a closely related problem. Notice also that the result is of the type covered by \cite{Bai66} since we essentially want to prove:
\begin{equation}
[X, \rmH^s]_\theta = \{ u \in [\rmL^2, \rmH^s]_\theta \ : \ \rmd u \in [\rmL^2, \rmH^{s-1}]_\theta \}.
\end{equation}
The interpolation spaces on the right are then identified as standard spaces, by the reiteration theorem (Theorem 26.3 in \cite{Tar07}).
\end{proof}

Now we choose $s \in ]0,1/2[$ close to $1/2$. Then we choose $\theta \in ]0,1[$ close to $1$ such that $\theta (1-s) \in ]0,1/2[$. Then we set $s_1 = \theta s$ and $s_2 =  \theta (1-s)$. By choosing $s$ and $\theta$ cleverly we may get $s_1$ and $s_2$ as close to $1/2$ as we wish, yet strictly smaller. We then introduce the following norms, corresponding to the interpolation spaces of the above Lemma.
\begin{equation}\label{eq:nss}
\| u\|_{s_1, s_2}^2 = \| u\|_{\rmH^{s_1}}^2 + \| \rmd u\|_{\rmH^{-s_2}}^2.
\end{equation}
These norms are well defined on $X_h$. Combining Propositions \ref{prop:iss} and \ref{prop:isd} we get:
\begin{proposition}\label{prop:isinterhd}
We have an estimate:
\begin{equation}
1 \cleq \infsup{u \in  \tilde X_h}{v\in \tilde X_h}{|a(u,v)|}{\|u\|_{s_1,s_2} \, \|v\|_{s_1, s_2}}. 
\end{equation}
\end{proposition}

In what follows we suppose $C$ is a zero order operator with smooth coefficients, such that $A + C$ is injective. The bilinear form associated with $C$ is denoted $c$. We denote by $b$ the bilinear form $a+c$. 

Applying Proposition \ref{prop:pertacm} we obtain:
\begin{proposition}\label{prop:impstab}
We have an estimate:
\begin{equation}
1 \cleq \infsup{u \in X_h}{v\in  X_h}{|b(u,v)|}{\|u\|_{s_1,s_2} \, \|v\|_{s_1, s_2}}. 
\end{equation}
\end{proposition}

Let $Z$ be the space defined by the norm (\ref{eq:nss}). Let $Z'$ be the space with norm:
\begin{equation}
\| u\|_{s_2, s_1}^2 = \| u\|_{\rmH^{s_2}}^2 + \| u\|_{\rmH^{-s_1}}^2.
\end{equation}
Then $b$ is continuous on $Z' \times Z$, and it may be checked that in fact $Z'$ is the dual of $Z$ with respect to $b$. By construction $s_1 \leq s_2$, so $Z' \subseteq Z$ with continuous injection. Notice the crucial fact that $X_h$ is a subspace of both $Z$ and $Z'$.

Define the Galerkin projection $R_h: Y \to X_n$ by:
\begin{equation}
\forall v \in X_h \quad b(R_h u, v) = b(u,v).
\end{equation}
Proposition \ref{prop:isinterhd} shows that $R_h$ is uniformly stable $Z' \to Z$. This stability is then combined with the approximation orders proved in Propositions \ref{prop:ratepi} and \ref{prop:ratepid}, to give:
\begin{proposition}\label{prop:rhrate}
 We have:
\begin{equation}
\| u - R_h u \|_{s_1, s_2}  \leq h^{p+ s_1} | \nabla^{p+1} u |.
\end{equation}
\end{proposition}
\begin{proof}
We have:
\begin{align}\label{eq:galerr}
\| u - R_h u \|_{s_1, s_2} & \leq \| u - \Pi_h u \|_{s_1, s_2} + \| R_h (u - \Pi_h u) \|_{s_1, s_2},\\
& \cleq \| u - \Pi_h u \|_{s_1, s_2} + \| u - \Pi_h u \|_{s_2, s_1},\\
& \cleq \max \{ h^{p+1 - s_1}, h^{p+ s_2}, h^{p+1 - s_2}, h^{p+ s_1} \} | \nabla^{p+1} u |.
\end{align}
We notice that $1-s_2 \geq 1- s_2 \geq s_2 \geq s_1$, so we end up with the announced estimate.
\end{proof}
This convergence order is as close to $h^{p + 1/2}$ as we wish. The convergence order is therefore optimal except for a loss of $h^{-\epsilon}$, for $\epsilon >0$ arbitrarily small.

We can also get approximation orders in $\rmL^2$ by an Aubin-Nitsche trick:
\begin{proposition} We have:
\begin{equation}
| u - R_h u | \cleq  h^{p+ 2 s_1} | \nabla^{p+1} u |.
\end{equation}
\end{proposition}
\begin{proof}
We write:
\begin{align}
| u - R_h u | & \cleq \sup_{v \in \rmH^1} \frac{|b(u - R_hu ,v)|}{\| v \|_{\rmH^1}},\\
& \cleq \sup_{v \in \rmH^1} \frac{|b(u - R_hu ,v - R_hv )|}{\| v \|_{\rmH^1}},\\
& \cleq \sup_{v \in \rmH^1} \frac{\| u - R_hu\|_{Z'} \| v - R_hv \|_Z}{\| v \|_{\rmH^1}}.
\end{align}
Then we use Proposition \ref{prop:rhrate}.
\end{proof}

If we define $K$ and $K_h$ as in (\ref{eq:akdef}, \ref{eq:akhdef}), but with respect to the bilinear form $b$ instead of $a$, we may use the identity $K -K_n = (I - R_h) K$. From the convergence rates obtained for $R_h$, convergence rates for the eigenvectors and values are obtained by the techniques of \cite{Cha83}\cite{BabOsb89}\cite{BanOsb90}\cite{BabOsb94}.

\paragraph{Outlook.} We finish with some remarks on possible future work. The generality of the above abstract framework seems particularly adapted to the following extensions. 
\begin{itemize}

\item Firstly, one could be interested in a variable background Riemannian metric. A technique to obtain a fully discrete theory, would be to approximate the metric by discrete metrics, such as those defined by Regge \cite{Reg61}\cite{Chr11NM}. This approximation step should fall under the scope of the abstract consistency requirement we defined. This topic would be interesting to pursue on general manifolds, not just toruses.

\item Secondly, the framework also sheds light on second order equations, and non-conforming discretizations thereof. Let $\nabla$ be the Levi-Civita connection associated with a Riemannian metric. Recall that $\nabla^\star \nabla$, acting on differential forms, differs from the Hodge-Laplace by lower order curvature terms. The eigenvalue problem for $\nabla^\star \nabla$ could be approached using Crouzeix-Raviart elements. Remarkably, these are well defined in Regge calculus.

\item Thirdly, we have already motivated our results by applications to electromagnetic fields. The continuous Dirac equation has a gauge symmetry, associated with adding gradients to the magnetic vector potential. It would seem interesting to try to achieve a consistent gauge invariant discrete method. On the other hand this does not seem as important for the Dirac equation as for the Yang-Mills equations, the difference being the size of the kernel, which is infinite-dimensional in the second case.
\end{itemize}


\section{Acknowledgements}
Interesting discussions with Tore G. Halvorsen and Torquil M. S\o rensen are gratefully acknowledged, in particular on the ''fermion doubling problem''.

Key results were obtained while visiting University Paris 6, \'Ecole Normale Sup\'erieure and Institut Henri Poincar\'e, in spring 2013 and autumn 2015. The author is grateful for financial support, stimulating environments and kind hospitality.

This research was supported by the European Research Council through the FP7-IDEAS-ERC Starting Grant scheme, project 278011 STUCCOFIELDS.

\bibliography{../Bibliography/alexandria,../Bibliography/newalexandria,../Bibliography/mybibliography}{}
\bibliographystyle{plain}

\end{document}